\documentclass[a4paper, 12pt]{amsart}

\usepackage{color}
\usepackage{graphicx}
\usepackage{subfigure}
\usepackage{float}
\usepackage{url}
\usepackage{enumerate}
\usepackage{xspace}
\usepackage{multirow}
\usepackage{hhline}
\usepackage{lmodern}
\usepackage[justification=centering, font=small, labelfont=bf]{caption}
\usepackage{soul}

\newcommand{\df}[1]{\ensuremath{\mathrm{df}({#1})}\xspace}
\newcommand{\core}{{\mathrm{core}}}

\newcommand{\head}[1]{\vskip3pt\noindent{}\textbf{#1}.\xspace}

\theoremstyle{plain}
\newtheorem{theorem}{Theorem}[section]
\newtheorem{lemma}[theorem]{Lemma}
\newtheorem{proposition}[theorem]{Proposition}

\theoremstyle{definition}

\newtheorem{remark}[theorem]{Remark}

\newtheorem{example}[theorem]{Example}
\newtheorem*{example*}{Example}
\newtheorem{problem}{Problem}
\newtheorem{conjecture}{Conjecture}

\title{Cubic graphs with colouring defect 3}
\author[J. Karab\'a\v{s}]{J\'an Karab\'a\v{s}}
\email[J. Karab\'a\v{s}]{jan.karabas@umb.sk}
\author[E. M\'a\v cajov\'a]{Edita M\'a\v cajov\'a}
\email[E. M\'a\v cajov\'a]{macajova@dcs.fmph.uniba.sk}
\author[R. Nedela]{Roman Nedela}
\email[R. Nedela]{nedela@savbb.sk}
\author[M. \v Skoviera]{Martin \v Skoviera}
\email[M. \v Skoviera]{skoviera@dcs.fmph.uniba.sk}

\address[J. Karab\'a\v{s}]{
Department of Computer Science, Faculty of Natural Sciences,
Matej Bel University, Bansk\'a Bystrica, Slovakia
}

\address[E. M\'a\v cajov\'a, M. \v Skoviera]{
Comenius University, Mlynsk\' a dolina, Bratislava, Slovakia
}

\address[R. Nedela]{
Faculty of Applied Sciences, University of West Bohemia,
Pilsen, Czech Republic}

\address[J. Karab\'a\v{s}, R. Nedela]{
Mathematical Institute of Slovak Academy of Sciences, Bansk\'a Bystrica, Slovakia
}

\begin{document}

\begin{abstract}
The colouring defect of a cubic graph is the smallest number of
edges left uncovered by any set of three perfect matchings.
While $3$-edge-colourable graphs have defect $0$, those that
cannot be $3$-edge-coloured (that is, snarks) are known to have
defect at least $3$. In this paper we focus on the structure
and properties of snarks with defect~$3$. For such snarks we
develop a theory of reductions similar to standard reductions
of short cycles and small cuts in general snarks. We prove that 
every snark with defect~$3$ can be reduced to a snark with 
defect $3$ which is either nontrivial (cyclically 
$4$-edge-connected and of girth at least~$5$) or to one that 
arises from a nontrivial snark of defect greater than $3$ by 
inflating a vertex lying on a suitable $5$-cycle to a triangle. 
The proofs rely on a detailed analysis of Fano flows associated 
with triples of perfect matchings leaving exactly three uncovered 
edges. In the final part of the paper we discuss application 
of our results to the conjectures of Berge and Fulkerson, 
which provide the main motivation for our research.
\end{abstract}
\subjclass[2020]{05C15 (Primary) 05C75 (Secondary)}
\maketitle

\section{Introduction}
\noindent{}Every $3$-edge-colourable cubic graph has a set of
three perfect matchings that cover all of its edges.
Conversely, three perfect matchings cover the edge set of a
cubic graph only when they are pairwise disjoint and,
therefore, the graph is $3$-edge-colourable. It follows that if
a cubic graph is not $3$-edge-colourable, then any collection
of three perfect matchings leaves some of its edges not
covered. The minimum number of edges of a cubic graph $G$ left
uncovered by any set of three perfect matchings will be called
the \emph{colouring defect} of $G$ and will be denoted by
$\df{G}$. For brevity, we usually drop the adjective
``colouring'' and speak of the \emph{defect} of a cubic graph.
Clearly, a cubic graph has defect zero if and only if it is
$3$-edge-colourable, so defect can be regarded as a measure of
uncolourability of cubic graphs.

The concept of colouring defect was introduced by Steffen
\cite{S2} who used the notation $\mu_3(G)$ but did not coin any
term for it. Among other things he proved that every
$2$-connected cubic graph which is not $3$-edge-colourable --
that is, a \emph{snark} -- has defect at least $3$. He also
proved that the defect of a snark is at least as large as one
half of its girth. Since there exist snarks of arbitrarily
large girth \cite{Ko}, there exist snarks of arbitrarily large
defect.

The defect of a cubic graph was further examined by Jin and
Steffen in \cite{JS} and was also discussed in the survey of
uncolourability measures by Fiol et al.
\cite[pp.13-14]{FMS-survey}. In \cite{JS}, Jin and Steffen
studied the relationship of defect to other measures of
uncolourability, in particular its relationship to oddness. The
\emph{oddness} of a cubic graph $G$, denoted by $\omega(G)$, is
the minimum number of odd circuits in a $2$-factor of $G$; it
is correctly defined for any bridgeless cubic graph. Jin and
Steffen proved \cite[Corollary 2.4]{JS} that $\df G\geq
3\omega(G)/2$ and investigated the extremal case where $\df G=
3\omega(G)/2$ in detail.

In this paper we to continue the study of the colouring defect
of snarks with emphasis on snarks with minimum possible defect,
that is, defect $3$. Snarks whose defect equals~$3$ have a
remarkable property that they contain a $6$-cycle
\cite[Corollary~2.5]{S2}, which immediately implies that their
cyclic connectivity does not exceed $6$. This fact relates the
study of colouring defect to a fascinating conjecture of Jaeger
and Swart \cite[Conjecture 2]{JSw} which suggests that the
cyclic connectivity of every snark is bounded above by $6$. It
is therefore a natural question to ask what structural
properties of snarks ensure that their defect is~$3$, and
conversely, what structural properties are implied by the fact
that the defect is, or is not, this minimum possible value.

A natural approach to improving our understanding of the
structure of snarks with defect $3$ is through eliminating
certain trivial features that they might posses. The main
purpose of this paper is, therefore, to develop a theory of
reductions for snarks with defect~$3$ analogous to standard
reductions of short cycles and small cuts in general snarks.
Snarks that have cycle-separating edge-cuts of size smaller
than $4$ or circuits of length smaller than $5$ are generally
considered to be trivial. This is explained by the fact that if
a snark contains a digon, a triangle, or a quadrilateral, one
can easily remove it and subsequently restore $3$-regularity to
produce a smaller snark; similar reductions can be applied to
small cuts~\cite{Wat}. Thus, a \emph{nontrivial} snark must be
cyclically $4$-edge-connected and have girth at least $5$.

The standard reductions to nontrivial snarks have been
extremely useful in numerous investigations related to
important conjectures in the area, such as the cycle
double-cover conjecture \cite{Jae-strong}, 5-flow conjecture
\cite{Ko2}, or Fulkerson's conjecture \cite{MM}. In this
situation it is natural to attempt finding similar reductions
within the class of snarks with defect 3. The expected aim
would be to show that, given a snark with defect~3, one can
eliminate cycles of length smaller than~5 and cycle-separating
edge-cuts of size smaller than 4 to produce -- in a certain
natural manner -- a snark whose defect is still 3 but lacks
these ``trivial'' features. Our main results,
Theorems~\ref{thm:reduction} and~\ref{thm:essential},
demonstrate that this expectation almost comes true.

In Theorem~\ref{thm:reduction} we show that every snark $G$
with defect $3$ can be reduced to a snark $G'$ with defect $3$ 
such that either $G'$ is nontrivial or $G'$ contains a single 
triangle whose contraction produces a nontrivial snark with 
defect greater than~$3$. Such a triangle is called 
\emph{essential}.

In Theorem~\ref{thm:essential} we further show that the reduced
snark $G'$ arises from a nontrivial snark by inflating a vertex
lying on a 5-cycle that contains an edge $uv$ such that
$G'-\{u,v\}$ is $3$-edge-colourable. Our main results thus
indicate that in the study of colouring defect of cubic graphs
one cannot completely avoid graphs with triangles.

In addition, in Theorem~\ref{thm:df_n3} we show that by
contracting an essential triangle in a snark with defect $3$
one can obtain a nontrivial snark with an arbitrarily high
defect.

Proofs of these results require a detailed study of triples
$\{M_1,M_2,M_3\}$ of perfect matchings in cubic graphs that
leave the minimum number of uncovered edges, along with
structures derived from them, especially hexagonal cores and
Fano flows. A \emph{hexagonal core} of a snark with defect $3$
is a $6$-cycle induced by the set of all edges that are not
simply covered by the triple $\{M_1, M_2,M_3\}$; it alternates
the uncovered edges with the doubly covered ones. A closely
related concept is a that of a \emph{Fano flow}. It is a
nowhere-zero
$\mathbb{Z}_2\times\mathbb{Z}_2\times\mathbb{Z}_2$-flow on $G$
induced by the $2$-factors complementary to $M_1$, $M_2$, and
$M_3$. Its flow values can be identified with the points of the
Fano plane and flow patterns around vertices with the lines of
a configuration $F_4$ of four lines covering all seven points
of the Fano plane (see Figure~\ref{fig:konf}). Analysing Fano
flows across small edge-cuts or around short circuits makes up
a substantial part of the proofs of
Theorems~\ref{thm:reduction} and~\ref{thm:essential}.

Our paper is organised as follows. In the next section we
collect the most important definitions and facts needed for
understanding this paper. In Section~\ref{sec:arrays} we
introduce structures related to the colouring defect and
investigate their properties.
After establishing auxiliary
results about reductions in Section~\ref{sec:reduction}, we
prove our main results, Theorems~\ref{thm:reduction}
and~\ref{thm:essential}, in Section~\ref{sec:main}. In
Section~\ref{sec:Berge} we discuss how the reductions
established in the previous sections can be applied to
verifying Berge's conjecture for snarks of defect~$3$. The
conjecture states that five perfect matchings are sufficient to
cover all edges of any bridgeless cubic graphs. We explain why
every snark of defect $3$ fulfils Berge's conjecture, moreover,
we provide a structural characterisation of  those snarks of
defect 3 that require five perfect matchings to cover their
edges (Theorem~\ref{thm:bergegen}). This significantly
strengthens a result of Steffen \cite[Theorem~2.14]{S2}, where
Berge's conjecture was verified for cyclically
$4$-edge-connected cubic graphs of defect $3$. The proof will
appear in a separate article \cite{KMNS-Bergegen} (see
also~\cite{KMNS-eurocomb}). Finally, in Section~\ref{sec:comp}
we summarise the outputs of computer-aided experiments directed
towards defect and cores of nontrivial snarks of order up to
36. These results provide partial support for several
conjectures that we propose at the end of this paper.

\section{Preliminaries}\label{sec:prelim}

\noindent{}All graphs in this paper are finite and for the most
part cubic (3-valent). Multiple edges and loops are permitted.
We use the term \emph{circuit} to mean a connected $2$-regular
graph. An $m$-\emph{cycle} is a circuit of length~$m$. The
length of a shortest circuit in a graph is its \emph{girth}. If
$H$ is an induced subgraph of $G$, we let $G/H$ denote the
graph arising from $G$ by contracting each component of $H$
into a single vertex.

For a subgraph  (or just a set of vertices) $Y$ of a graph $G$
we let $\delta(Y)$ denote the edge cut consisting all edges
joining $Y$ to vertices not in $Y$. A connected graph $G$ is
said to be \emph{cyclically $k$-edge-connected} if the removal
of fewer than $k$ edges from $G$ cannot create a graph with at
least two components containing circuits. An edge cut $S$ in
$G$ that separates two circuits from each other is
\emph{cycle-separating}.

Large graphs are typically constructed from smaller building
blocks called multipoles. Like a graph, each \emph{multipole}
$M$ has its vertex set $V(M)$, its edge set $E(M)$, and an
incidence relation between vertices and edges. Each edge of $M$
has two ends, and each end may, but need not be, incident with
a vertex of $M$. An end of an edge that is not incident with a
vertex is called a \emph{free end} or a \emph{semiedge}. An
edge with exactly one free end is called a \emph{dangling
edge}. An \emph{isolated edge} is one whose both ends are free.
All multipoles considered in this paper are \emph{cubic}; this
means that every vertex is incident with exactly three
edge-ends. An \emph{$n$-pole} is a multipole with $n$ free
ends. Free ends of a multipole can be distributed into pairwise
disjoint sets, called \emph{connectors}.
An \emph{$(n_1, n_2,\ldots, n_k)$-pole} is an $n$-pole with
$n=n_1+n_2+\cdots + n_k$ whose semiedges are distributed into
$k$ connectors $S_1,S_2,\ldots, S_k$, each $S_i$ being of size
$n_i$.

Consider an arbitrary $n$-pole $M$ and choose two distinct free
ends $s_i$ and $s_j$ belonging to edges $e$ and $e'$,
respectively. We say that a multipole $M'$ is formed by the
\emph{junction} of $s_i$ and $s_j$ if $M'$ arises from $M$ by
identifying $s_i$ and $s_j$ while retaining the other ends of
$e$ and $e'$. The newly formed edge is denoted by $s_i*s_j$. If
$s_i$ and $s_j$ are the free ends of the same isolated
edge $e$, the junction amounts to the deletion of $e$. A
\emph{junction} of two $n$-poles $M$ and $N$ is a cubic graph,
denoted by $M * N$, arising from $M$ and $N$ by performing the
junction of their respective semiedge sets.  If a bijection
$\sigma$ between the semiedges of $M$ and those of $N$ is
specified, we write $M*_{\sigma}N$. Throughout the paper we
will use the following convention: if a graph $G$ can be
expressed in the form $G=M*N$, then, depending on the context,
$G-M$ will either mean the multipole $N$ including its dangling
edges or will denote the induced subgraph $G-M$. There is no
danger of confusion.

An \emph{edge colouring} of a multipole $M$ is a mapping from
the edge set of $M$ to a set of colours. A colouring is
\emph{proper} if any two edge-ends incident with the same
vertex carry distinct colours. A \emph{$k$-edge-colouring} is a
proper colouring where the set of colours has $k$ elements. A
cubic graph $G$ is said to be \emph{colourable} or
\emph{uncolourable} depending on whether if it does or does not
admit a $3$-edge-colouring. A $2$-connected uncolourable cubic
graph is called a \emph{snark}.

Our definition leaves the concept of a snark as wide as
possible since more restrictive definitions may lead to
overlooking certain important phenomena in cubic graphs.
Our definition thus follows Cameron et al. \cite{CCW},
Nedela and \v Skoviera \cite{NS-decred}, Steffen \cite{S1}, and
others, rather than a more common approach where snarks are
required to be cyclically $4$-edge-connected and have girth at
least $5$, see for example~\cite{FMS-survey}. In this paper,
such snarks are called \emph{nontrivial}.

The problem of nontriviality of snarks has been widely
discussed in the literature, see for example
\cite{CCW,NS-decred,S1}. Here we follow a systematic approach
to nontriviality of snarks proposed by Nedela and \v Skoviera
\cite{NS-decred}. Aset of vertices or an induced subgraph $H$
of a snark $G$ is called \emph{non-removable} if $G - V(H)$ is
colourable; otherwise, $H$ is \emph{removable}. A snark $G$ is
\emph{critical} if every pair of distinct adjacent vertices of
$G$ is non-removable. A snark is \emph{bicritical} if every
pair of distinct vertices of $G$ is non-removable. A snark $G$
is \emph{irreducible} if every induced subgraph $H$ with at
least two vertices is non-removable. It is known that a snark
is irreducible if and only if it is bicritical,
see~\cite[Theorem~4.4 and Corollary~4.6]{NS-decred}.

It is an easy observation that short cycles in snarks
are removable.
\begin{lemma}\label{lem:odstr4}
In an arbitrary snark, every circuit of length at most $4$ is
removable.
\end{lemma}

In the study of snarks it is useful to take the colours $1$,
$2$, and $3$ to be the nonzero elements of the group
$\mathbb{Z}_2\times\mathbb{Z}_2$. To be specific, one can
identify a colour with its binary representation: $1=(0,1)$,
$2=(1,0)$, and $3=(1,1)$. With this choice, the condition that
the three colours meeting at any vertex $v$ are all distinct is
equivalent to requiring that the sum of the colours at $v$ is
$0=(0,0)$. The latter condition coincides with the Kirchhoff
law for a nowhere-zero $\mathbb{Z}_2\times\mathbb{Z}_2$-flow on
a graph, or a multipole. Thus a proper $3$-edge-colouring of a
cubic multipole coincides the a nowhere-zero
$\mathbb{Z}_2\times\mathbb{Z}_2$-flow on it.

The following well-known statement is a direct consequence of
Kirchhoff's law.

\begin{lemma}[Parity Lemma]\label{lem:par}
Let $M = M(s_1,s_2, \ldots,s_k)$ be a $k$-pole endowed with a
$3$-edge-colouring $\phi$. Then
$$\sum_{i=1}^k\phi(s_i)=0.$$
Equivalently, the number of free ends of $M$ carrying any fixed
colour has the same parity as $k$.
\end{lemma}

Parity Lemma has a remarkable consequence that colourable
$4$-poles extendable to a snark fall into two types --
isochromatic and heterochromatic (see for example
\cite[Section~3]{ChS}). Such a $4$-pole is \emph{isochromatic}
if its semiedges can be partitioned into two pairs $\{x,x'\}$
and $\{y,y'\}$ such that for every $3$-edge-colouring $\phi$
one has $\phi(x)=\phi(x')$ and $\phi(y)=\phi(y')$; otherwise it
is \emph{heterochromatic}. It can be shown that a colourable
$4$-pole is isochromatic $M$ if and only if $G$ arises from a
snark $G$ by removing two adjacent vertices $u$ and $v$.
Moreover, the two pairs $\{x,x'\}$ and $\{y,y'\}$ correspond to
the edges formerly incident with the vertices $u$ and $v$,
respectively.

\section{Arrays of perfect matchings and the defect of a
snark}\label{sec:arrays}

\noindent{}Various important structures in cubic graphs, such
as Tait colourings, Berge or Fulkerson covers, Fan-Raspaud
colourings, and several others, can be described in terms of
sets or lists of perfect matchings obeying certain additional
conditions. In this paper such lists will be called
\emph{arrays} of perfect matchings. To be more precise, a
\emph{$k$-array of perfect matchings} in a cubic graph $G$,
briefly a \emph{$k$-array} for $G$, is an arbitrary collection
$\mathcal{M}=\{M_1, M_2, \ldots, M_k\}$ of $k$ not necessarily
distinct perfect matchings of $G$. In particular, a \emph{Berge
cover} of $G$ is a $5$-array $\mathcal{B}$ such that each edge
of $G$ belongs to some member of $\mathcal{B}$; a
\emph{Fulkerson cover} is $6$-array $\mathcal{F}$ such that
each edge of $G$ belongs to precisely two members of
$\mathcal{F}$.

This paper focuses on the properties of snarks that can be
expressed by means of $k$-arrays of perfect matchings with
$k=3$. Since every proper $3$-edge-colouring gives rise to
an array whose members are the three colour classes, $3$-arrays
can be regarded as approximations of $3$-edge-colourings. An
edge of $G$ that belongs to at least one of the perfect
matchings of the array $\mathcal{M}=\{M_1, M_2, M_3\}$ will be
considered to be \emph{covered}. An edge will be called
\emph{uncovered}, \emph{simply covered}, \emph{doubly covered},
or \emph{triply covered} if it belongs, respectively, to zero,
one, two, or three distinct members of~$\mathcal{M}$.

Given a $3$-array $\mathcal{M}$ of $G$, it is a natural task to
maximise the number of covered edges, or equivalently, to
minimise the number of uncovered ones. A $3$-array that leaves
the minimum number of uncovered edges will be called
\emph{optimal}. The minimal number of edges left uncovered by an
optimal $3$-array is the \emph{colouring defect} of $G$,
briefly, the \emph{defect}, denoted by $\df{G}$.

In this section we establish several basic results concerning
$3$-arrays and the defect of a cubic graph with emphasis on
optimal $3$-arrays. We remark that some of the ideas and
results have already appeared in the papers by Steffen and
others~\cite{JS, JSM, S2, S3}. However, in order to make our
exposition self-contained we need to accompany such results
with proofs.

With each $3$-array of perfect matchings one can associate
several important structures which reside within the underlying
graph. We discuss two of them: the characteristic flow and the
core.

\head{1. Characteristic flow} Let $\mathcal{M}=\{M_1, M_2,
M_3\}$ be a $3$-array of a cubic graph $G$. One way to describe
$\mathcal{M}$ is based on regarding the indices $1$, $2$, and
$3$ as colours. Since the same edge may belong to more than one
member of $\mathcal{M}$, an edge of $G$ may receive from
$\mathcal{M}$ more than one element of the set
$\{1,2,3\}$. To each edge $e$ of $G$ we can therefore assign
the list $\phi(e)$ of elements of $\{1,2,3\}$  in
lexicographic order it receives from $\mathcal{M}$. We let
$w(e)$ denote the number of colours in the list $\phi(e)$
and call it the \emph{weight}.
In this way $\mathcal{M}$ gives rise to an edge-colouring
$$\phi\colon E(G)\to\{\emptyset, 1,2,3, 12, 13, 23, 123\}$$
where $\emptyset$ denotes the empty list. We call $\phi$ the
\emph{characteristic colouring} for $\mathcal{M}$. Obviously,
such a colouring determines a $3$-array if and only if, for
each vertex $v$ of $G$, all three indices from $\{1,2,3\}$
occur precisely once on the edges incident with $v$. In
general, $\phi$ need not be a proper colouring. As we shall
see below, $\phi$ is a proper edge-colouring if and only if
$G$ has no triply covered edge.

A different but equivalent way of representing a $3$-array uses
a mapping
\[
\chi\colon E(G)\to \mathbb{Z}_2^3,
\quad e\mapsto \chi(e)=(x_1, x_2, x_3)
\]
defined by setting $x_i = 0$ if and only if $e\in M_i$ for
$i\in\{1,2,3\}$. Since the complement of each $M_i$ in $G$ is a
$2$-factor, it is easy to see that $\chi$ is a
$\mathbb{Z}_2^3$-flow. We call $\chi$ \emph{the characteristic
flow} for $\mathcal{M}$. Again, $\chi$ is a nowhere-zero
$\mathbb{Z}_2^3$-flow if and only if $G$ contains no triply
covered edge. In the context of $3$-arrays the characteristic
flow was introduced in~\cite[p. 166]{JSM},
but the idea is older, see \cite[Section~4]{MS-TCS}. Observe that
the characteristic flow $\chi$ of a $3$-array and the colouring
$\phi$ uniquely determine each other. In particular, the
condition on $\phi$ requiring all three indices from
$\{1,2,3\}$ to occur precisely once in a colour around any
vertex is equivalent to Kirchhoff's law.

The following result characterises $3$-arrays with no triply
covered edge.

\begin{proposition}\label{prop:notriply}
Let $\mathcal{M}$ be a $3$-array of perfect matchings of a
cubic graph $G$. The following three statements are equivalent.
\begin{enumerate}[{\rm (i)}]
\item $G$ has no triply covered edge with
    respect to $\mathcal{M}$.
\item The associated colouring $\phi\colon
    E(G)\to\{\emptyset, 1,2, 3, 12, 13, 23, 123\}$ is
    proper.
\item The characteristic flow $\chi$ for $\mathcal{M}$,
    with values in $\mathbb{Z}_2^3$, is nowhere-zero.
\end{enumerate}
\end{proposition}

\begin{proof}
(i) $\Leftrightarrow$ (ii): Consider an arbitrary vertex $v$ of
$G$ and the three edges $e_1$, $e_2$, and $e_3$ incident with
$v$. Since every perfect matching contains an edge incident
with $v$, and since $G$ has no triply covered edge, the
distribution of weights on $(e_1,e_2,e_3)$ is either $(1,1,1)$
or $(2,1,0)$ up to permutation of values. In both cases it is
obvious that $e_1$, $e_2$, and $e_3$ receive distinct colours
from $\phi$. For the converse, if an edge $e=uv$ of $G$ is
triply covered, then $\phi(e)=123$ and the other two edges
incident with $u$ receive colour $\emptyset$. Thus $\phi$ is
not proper.

\medskip

(ii) $\Leftrightarrow$ (iii):  This is an immediate consequence
of the fact that the number of vanishing coordinates in
$\chi(e)$ coincides with $w(e)$.
\end{proof}

\begin{figure}[h!]
\subfigure[]{ \includegraphics[scale=1.4]{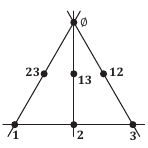}\label{fig:konf1}
}\qquad
\subfigure[]{ \includegraphics[scale=1.4]{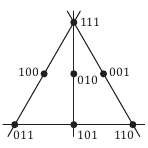}\label{fig:konf2} }
\caption{The configuration $F_4$ for $3$-arrays with no triply covered edges }
\label{fig:konf}
\end{figure}

The fact that $3$-arrays with no triply covered edge are
associated to certain nowhere-zero flows and proper
edge-colourings suggests that they deserve a special attention.
If a $3$-array $\mathcal{M}$ has no triply covered edge, the
flow values of $\chi$ can be regarded as points of the Fano
plane $PG(2,2)$ represented by the standard projective
coordinates from $\mathbb{Z}_2^3-\{0\}$. In this
representation, the points of $PG(2,2)$ are the non-zero
elements of $\mathbb{Z}_2^3$ and the lines of $PG(2,2)$ are the
$3$-element subsets $\{x,y,z\}$ of $\mathbb{Z}_2^3-\{0\}$ such
that $x+y+z=0$. The definition of the characteristic flow
implies that for every vertex $v$ of $G$ the three flow values
meeting at $v$ form a line of the Fano plane. Since every
vertex of $G$ is incident with an edge of each member of
$\mathcal{M}$, one can easily deduce that only four lines of
the Fano plane can occur around a vertex, and these lines form
a point-line configuration in $PG(2,2)$ which is shown in
Fig.~\ref{fig:konf}; we denote this configuration by $F_4$. On
the left-hand side of the figure the points are labelled with
the corresponding colours -- subsets of $\{1,2,3\}$. Based on
this correspondence we will refer to the colouring $\phi$
and the flow $\chi$ as the \emph{Fano colouring} and the
\emph{Fano flow} associated with a $3$-array $\mathcal{M}$. The
colours from $\{\emptyset, 1,2, 3, 12, 13, 23\}$ and the
corresponding elements of $\mathbb{Z}_2^3-\{0\}$ will be used
interchangeably.

\head{2. Core} Another important structure associated with a
$3$-array is its core. The \emph{core} of a $3$-array
$\mathcal{M}=\{M_1, M_2, M_3\}$ of $G$ is the subgraph of $G$
induced by all the edges of $G$ that are not simply covered; we
denote it by $\core(\mathcal{M})$. The core will be called
\emph{optimal} whenever $\mathcal{M}$ is optimal.

The edge set of $\core(\mathcal{M})$ thus coincides with
$E_0(\mathcal{M})\,\cup\,E_{23}(\mathcal{M})$, where
$E_0(\mathcal{M})$ and $E_{23}(\mathcal{M})$ denote the set of
all uncovered edges  and the set of and doubly and triply
covered edges with respect to $\mathcal{M}$, respectively.
If $|E_0(\mathcal{M})|=k$, we say that
$\core(\mathcal{M})$ is a \emph{$k$-core}. It is worth
mentioning that if $G$ is $3$-edge-colourable and $\mathcal{M}$
consists of three disjoint perfect matchings, then
$\core(\mathcal{M})$ is empty. If $G$ is not $3$-edge-colourable, then
every core must be nonempty. Figure~\ref{fig:petersen_core}
shows the Petersen graph endowed with a $3$-array whose core is
the ``outer'' $6$-cycle. The hexagon is in fact an optimal core
of the Petersen graph.


\begin{figure}[h!]
 \centering
 \includegraphics[scale=1.4]{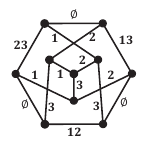}
 \caption{An optimal $3$-array of the Petersen graph}
\label{fig:petersen_core}
\end{figure}

The following proposition, much of which was proved by Steffen
in {\cite[Lemma~2.2]{S2}}, lists the most fundamental
properties of cores. We include the proof for the reader's
convenience.

\begin{proposition}\label{prop:core}
Let $\mathcal{M}=\{M_1, M_2, M_3\}$ be an arbitrary $3$-array
of perfect matchings of a snark $G$. Then the following hold:
\begin{enumerate}[{\rm (i)}]
\item Every component of $\core(\mathcal{M})$ is
    either an even circuit which alternates doubly covered
    and uncovered edges or a subdivision of a cubic graph.
    Moreover, $\core(\mathcal{M})$ is a collection of
    disjoint circuits if and only if $G$ has no triply
    covered edge.

\item Every $2$-valent vertex of $\core(\mathcal{M})$ is
    incident with one doubly covered edge and one uncovered
    edge, while every $3$-valent vertex is incident with
    one triply covered edge and two uncovered edges.
\item $E_{23}(\mathcal{M})$ is a perfect matching of
    $\core(\mathcal{M})$; consequently,
    $|E_0| = |E_2(\mathcal{M})| +2|E_3(\mathcal{M})|$.
\item $G-E_0(\mathcal{M})$ is
    $3$-edge-colourable.
\item If $\mathcal{M}$ is optimal, then every component of
    $\core(\mathcal{M})$ is a simple graph which is either an even circuit of
    length at least $6$ or a subdivision of a cubic graph.
\end{enumerate}
\end{proposition}

\begin{proof}
Set $H=\core(\mathcal{M})$. We claim that every vertex of $H$
has valency at least $2$. Indeed, if a vertex $v$ of $G$ is
incident with a triply covered edge, then the other two edges
incident with $v$ are uncovered. So, in this case, $v$ has
degree $3$ in $H$. If $v$ is incident with a doubly covered
edge, then exactly one of the remaining edges is simply covered
and the other one is uncovered. Thus $v$ has valency $2$ in
$H$. If all three edges incident with $v$ are simply covered,
then $v$ does not belong to $H$. It follows that every vertex
in $H$ has degree at least $2$, and also that the set
$E_{23}(\mathcal{M})$ of edges forms a perfect matching in $H$.
Statements (i)--(iii) now follow immediately.

To prove (iv), assign colour $i$ to every edge of $G$ that is
simply covered and belongs to $M_i$. If an edge of $G$ is
doubly covered, then both end-vertices are incident in $G$ with
one uncovered edge and one simply covered edge. It follows that
we can colour such an edge with the colour $c\in \{1,2,3\}$
that does not occur on the simply covered edges adjacent to it.
Finally, if an edge of $G$ is triply covered, both end-vertices
are incident with two uncovered edges. Thus we can colour such
an edge with any colour from $\{1,2,3\}$. In this manner we
have clearly produced a $3$-edge-colouring of
$G-E_0(\mathcal{M})$.

To prove (v) it is enough to argue that neither a digon nor a
$4$-cycle can occur as components of the core of an optimal
$3$-array. If $\core(\mathcal{M})$ contains a $4$-cycle
$Q=(e_0e_1e_2e_3)$, then two edges of $Q$ are uncovered, say
$e_0$ and $e_2$, and the other two are doubly covered. Clearly,
one of the three perfect matchings covers both $e_1$ and $e_3$.
Without loss of generality we may assume that it is $M_1$. The
characteristic colouring $\phi$ then satisfies
$\phi(e_1)=1i$ and $\phi(e_3)=1j$ for some
$i,j\in\{2,3\}$. We can modify $\phi$ to $\phi'$ by
setting $\phi'(e_0)=1$, $\phi'(e_2)=1$,
$\phi'(e_1)=i$, and $\phi'(e_0)=j$, and leaving all the
remaining edges $\phi'(e)=\phi(e)$. However, $\phi'$
now clearly determines a $3$-array with fewer uncovered edges,
contradicting the minimality of~$\mathcal{M}$. This proves that
$\core(\mathcal{M})$ does not contain a quadrilateral. The
argument that $\core(\mathcal{M})$ does not contain a digon is
similar and therefore is omitted.
\end{proof}

We say that a $3$-array $\mathcal{M}$ of $G$ has a
\emph{regular core} if each component of $\core(\mathcal{M})$
is a circuit; otherwise the core is called \emph{irregular}. By
Proposition~\ref{prop:core}(ii), a core is regular if and only
if $G$ has no triply covered edge. (Steffen~\cite{S2} calls
such a core \emph{cyclic}, but we believe that the letter term
is somewhat misleading as it might suggest that the core is a
single $k$-cycle for some $k$.) The well-known conjecture of
Fan and Raspaud \cite{FR} states that every bridgeless cubic
graph has three perfect matchings $M_1$, $M_2$, and $M_3$ with
$M_1\cap M_2\cap M_3=\emptyset$. Equivalently, the conjecture
states that  every bridgeless cubic graph has a $3$-array with
a regular core. The conjecture is trivially true for
$3$-edge-colourable graphs. M\'a\v cajov\'a and \v Skoviera
\cite{MS-Comb} proved this conjecture to be true for cubic
graphs with oddness $2$. We emphasise that neither the
conjecture nor the proved facts suggest anything about optimal
cores.

The following theorem characterises snarks with minimal colour
defect. The lower bound $3$ for the defect of a snark is due to
Steffen~\cite[Corollary~2.5]{S2}.

\begin{theorem}\label{thm:main}
Every snark $G$ has $\df{G}\ge 3$. Furthermore, the following
three statements are equivalent for any cubic graph $G$.
\begin{enumerate}[{\rm(i)}]
\item $\df{G}=3$.
\item The core of any optimal $3$-array of $G$ is a
    $6$-cycle.
\item $G$ contains an induced $6$-cycle $C$ such that the
    subgraph $G-E(C)$ admits a proper $3$-edge-colouring
    under which the six edges of $\delta(C)$ receive
    colours $1,1,2,2,3,3$ or $1,2,2,3,3,1$ with respect to
    a fixed cyclic ordering induced by an orientation of
    $C$.
\end{enumerate}
\end{theorem}

\begin{proof}
Let $G$ be a snark. First observe that any two perfect
matchings in $G$ intersect. Indeed, if there were two disjoint
perfect matchings in $G$, the set of remaining edges would be a
third perfect matching, implying that $G$ is
$3$-edge-colourable.

Now consider an optimal $3$-array
$\mathcal{M}=\{M_1,M_2,M_3\}$ of perfect matchings of $G$, and
let $H$ be the core of~$\mathcal{M}$. Since $G$ is a snark, we
have $\df{G}>0$, so $G$ contains at least one uncovered edge
and at least one multiply covered edge. To prove that
$\df{G}\ge 3$ we consider two cases.

\medskip\noindent
Case 1.\emph{ $G$ contains a triply covered edge $e$.}
Since $G$ is bridgeless, each end-vertex of $e$ is
incident with two distinct uncovered edges. By
Proposition~\ref{prop:core}(v), these four edges are all
distinct, implying that $\df{G}\ge 4$.

\medskip\noindent
Case 2. \emph{$G$ contains no triply covered edge.}
Proposition~\ref{prop:core}(i) and (v) now implies that each component
of $H$ is a circuit of length at least $6$, which means that
there are at least three uncovered edges. Hence, $\df{G}\ge 3$
in this case.

\medskip
So far we have shown that for every snark we have $\df{G}\ge
3$. We now finish the proof by proving that the statements
(i)-(iii) are equivalent. Let $G$ be an arbitrary cubic graph.
Assume that $\df{G}=3$. If we combine (ii) and (v) of
Proposition~\ref{prop:core}, we can conclude that $G$ contains
no triply covered edge. Proposition~\ref{prop:core}(i) now
tells us that each component of $H$ is an even circuit that
alternates uncovered and doubly covered edges. By
Proposition~\ref{prop:core}(v), each such circuit has has
length at least $6$, so $H$ must be a single hexagon. This
establishes the implication (i) $\Rightarrow$ (ii).

\medskip
(ii) $\Rightarrow$ (iii): Assume that the core of an optimal
$3$-array $\mathcal{M}$ of $G$ is a $6$-cycle $C$. Hence, $G$
is a snark. By Proposition~\ref{prop:core}(i) and (v),
$G$ has no triply covered edge. It follows that on $C$ the uncovered edges
and the doubly covered edges alternate and that the edges
leaving $C$ are simply covered.
Since $G$ is a snark,
any two perfect matchings intersect, which implies that the three
doubly covered edges receive colours $12$, $13$, and $23$ in
some order. The colours of doubly covered edges determine the
order of colours on the edges leaving $C$ uniquely, and it is
easy to see that they are as stated. Consequently, the
resulting structure is as illustrated in Figure~\ref{fig:core3}
up to permutation of the set $\{1,2,3\}$. In particular,
$C=(e_0e_1\ldots e_5)$, and the edges leaving $C$ are $f_0,
f_1, \ldots, f_5$, where each $f_i$ is adjacent to $e_{i-1}$
and $e_i$ with indices reduced modulo $6$.

It remains to show that $C$ is an induced $6$-cycle. Suppose
not. Then $C$ has a chord, necessarily a simply covered edge.
Without loss of generality we can assume that under the Fano
colouring $\phi_{\mathcal{M}}$ corresponding to $\mathcal{M}$
the chord has colour $1$. If we adopt the notation of
Figure~\ref{fig:core3}, the latter assumption means that
$f_4=f_5$. However, we can now extend the $3$-edge-colouring of
$G-E(C)$ determined by $\mathcal{M}$ to a $3$-edge-colouring
$\psi$ of the entire~$G$. Indeed, it is sufficient to set
$\psi(e_1)=1$, which further forces $\psi(e_0)=2$,
$\psi(e_2)=3$, and $\psi(e_3)=\psi(e_5)=1$. Thus we can define
$\psi(e_4)=2$ and $\psi(f_4)=3$ and let
$\psi(x)=\phi_{\mathcal{M}}(x)$ for all the remaining edges of
$G$. Clearly, this is a proper $3$-edge-colouring of $G$, which
is a contradiction.

The implication (iii) $\Rightarrow$ (i) is trivial.
\end{proof}

\begin{figure}[h]
\includegraphics[scale=1.5]{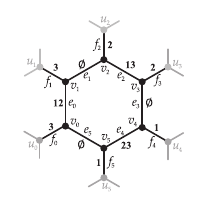}
\caption{The core for $\mathrm{df}=3$ and its vicinity}
\label{fig:core3}
\end{figure}

If $G$ is a snark with colouring defect $3$, then by
Theorem~\ref{thm:main}~(iii) it contains an optimal array
$\mathcal{M}=\{M_1,M_2,M_3\}$ whose core is an induced
$6$-cycle. Such a core will be referred to as a \emph{hexagonal
core} of $G$.

Consider an arbitrary induced $6$-cycle $Q=(q_0q_1 \dots q_5)$
in a snark $G$ with $\df{G}=3$, and let $r_1, \ldots, r_5$ be
the edges of $\delta(Q)$, where each $r_i$ is adjacent to
$q_{i-1}$ and $q_i$, with indices reduced modulo $6$. In
general, $Q$ need not constitute the core of any optimal
$3$-array for~$G$. If it does, we say that $Q$ is a \emph{core
hexagon}. In such a case, there is an optimal array
$\mathcal{M}$ for $G$ such that $E_0(\mathcal{M})$ consists of
three independent edges of $Q$. This can happen in two ways,
either $E_0(\mathcal{M})=\{q_0,q_2,q_4\}$ or
$E_0(\mathcal{M})=\{q_1,q_3,q_5\}$.

Assume that $E_0(\mathcal{M})=\{q_1,q_3,q_5\}$. Let
$\phi=\phi_{\mathcal{M}}$ be the Fano colouring of $G$
induced by $\mathcal{M}$. Since permuting the indices of the
perfect matchings $M_1$, $M_2$, and $M_3$ of $\mathcal{M}$ does
not essentially change the $3$-array, we can assume, without
loss of generality, that $\phi(q_0)=12$, $\phi(q_2)=13$,
and $\phi(q_4)=23$. In other words, if
$E_0(\mathcal{M})=\{q_1,q_3,q_5\}$, we can assume that around
$Q$ the Fano colouring $\phi$ is as shown in
Figure~\ref{fig:core3}, with each $q_i$ being identified
with~$e_i$. Similarly, if $E_0(\mathcal{M})=\{q_0,q_2,q_4\}$,
we can assume that the values of $\phi$ around $Q$ are
rotated one step counter-clockwise with respect to
Figure~\ref{fig:core3}. The conclusion is  that if $Q$ is a
core hexagon, then
the cyclic sequence
$(\phi(r_0),\phi(r_1),\ldots,\phi(r_5))$ of colours
leaving $Q$ is either $(3,3,2,2,1,1)$ or $(3,2,2,1,1,3)$,
possibly after permuting the set $\{1,2,3\}$.

The previous considerations imply that in order to decide
whether $Q$ is, or is not, a core hexagon, it is sufficient to
check whether at least one of the assignments $(3,3,2,2,1,1)$
or $(3,2,2,1,1,3)$ of colours to the edges of $\delta(Q)$
extends to a proper $3$-edge-colouring of $G-E(Q)$. If both
assignments extend, then $Q$ is the core of two optimal
$3$-arrays with disjoint sets of uncovered edges. In this case
we say that $Q$ is a \emph{double-core hexagon}. If only one of
the assignments extends, then we say that $Q$ is a
\emph{single-core hexagon}.

Observe that if a snark contains a double-core hexagon, then
the two corresponding $3$-arrays constitute a Fulkerson cover
of $G$, that is, a collection of six perfect matchings that
cover each edge precisely twice. The distribution of various
types of hexagons in snarks will be discussed in
Section~\ref{sec:comp}.

\section{Reduction to nontrivial snarks}\label{sec:reduction}
\noindent{}Let $G$ be a snark containing a $k$-edge-cut $R$
with $k\ge 2$, which decomposes $G$ into a junction $H * K$ of
two $k$-poles $H$ and $K$. If one of them, say $H$, is
uncolourable, we can extend $H$ to a snark $G'$ of order not
exceeding that of $G$ by joining $H$  with a suitable $k$-pole
$L$ of order $|L|\le |K|$ (possibly $L=K$). Following
\cite{NS-decred}, we call $G'$ a \emph{$k$-reduction} of $G$,
and say that $G$ can be \emph{reduced} to $G'$ along $R$.
With a slight abuse of terminology, we also say that $G'$
arises from $G$ by a \emph{reduction} with respect to $R$. More
generally, we say that a snark $G$ can be \emph{reduced} to a
snark $G'$ if there exists a sequence $G=G_0, G_1, \ldots,
G_t=G'$ of snarks such that for each $i\in\{0,\ldots, t-1\}$
the snark $G_i$ can be reduced to $G_{i+1}$ along some edge
cut. A reduction $G'$ of $G$ is said to be \emph{proper} if
$|G'| < |G|$. A reduction $G' = H*L$ of $G=H * K$ where $L$ has
the minimum possible order is called a \emph{completion} of $H$
to a snark. Observe that if $G'=H*L = \bar{H}$ is a completion
of $H$, then $|L|=0$ or $1$ depending on whether $k$ is even or
odd, respectively. Moreover, if $k\in\{2,3\}$, then $H$ has a
unique completion to a snark  up to isomorphism.

The aim of the next two sections is to prove that every snark
with defect $3$ can be reduced to a nontrivial snark with
defect $3$, or else it has a very specific structure. In this
section we gather auxiliary results needed for the proofs of
our main results, which will be presented in
Section~\ref{sec:main}.

A short reflection reveals that a reduction to a nontrivial
snark of defect $3$ is clearly not possible when the snark in
question contains a triangle whose contraction produces a snark
with defect greater than $3$. A triangle with this property
will be called \emph{essential}. We show that the existence of
an essential triangle is the only obstruction that prevents a
snark with defect $3$ from reduction, and that, in such a case,
there is only one essential triangle in the graph. As we shall
see later, an infinite family of snarks containing essential
triangles indeed exists (see Theorem~\ref{thm:df_n3}). One such
graph can be created from the graph in
Figure~\ref{fig:snark34_43} by inflating the central vertex $z$
to a triangle.

Throughout this section we use the following notation: $G$ is a
snark with $\df{G}=3$, $\mathcal{M}=\{M_1,M_3,M_3\}$ is an
optimal $3$-array of $G$ whose core $C$ is a 6-cycle, $\chi$ is
the characteristic flow for $\mathcal{M}$, and $\phi$ is the
associated Fano colouring. We further assume, without loss of
generality, that the Fano colouring and the names of edges in
the vicinity of $C$ are those as in Figure~\ref{fig:core3}.

\begin{lemma}\label{lem:essential}
Let $G$ be a snark with defect $3$, let $C$ be a hexagonal core
of $G$, and let $T$ be an arbitrary triangle in $G$. The
following statements hold.
\begin{itemize}
\item[{\rm (i)}] If $C\cap T=\emptyset$, then $T$ is not
    essential.

\item[{\rm (ii)}]  If $C\cap T\ne\emptyset$, then $C\cap T$
    consists of a single uncovered edge.

\item[{\rm (iii)}] Every hexagonal core intersects at most
    one triangle.

\item[{\rm (iv)}] $G$ has at most one essential triangle.
 \end{itemize}
\end{lemma}

\begin{proof}
Let $G$ be a snark with $\df{G}=3$, let $C$ be the core of an
optimal $3$-array $\mathcal{M}$ of~$G$, and let $T$ be an
arbitrary triangle in~$G$. Clearly, $G/T$ is a snark. We claim
that if $C\cap T=\emptyset$, then $T$ is not essential. Indeed,
if $C\cap T=\emptyset$, then $\delta(T)$ consists of simply
covered edges, and by the Kirchhoff law for the characteristic
flow these edges belong to three distinct members of
$\mathcal{M}$. It follows that $\mathcal{M}$ induces a
$3$-array $\mathcal{M}'$ of $G/T$ with $\core(\mathcal{M}')=C$.
Hence $\df{G/T}=3$, which means that the triangle $T$ is not
essential. This proves (i).

Now assume that $C\cap T\ne\emptyset$. Obviously, $C\cap T$
consists of a single edge $e$ because $C$ is an induced
$6$-cycle, by Theorem~\ref{thm:main}. If $e$ was doubly
covered, then $C\cap\delta(T)$ would consist of two uncovered
edges, which in turn would violate the Kirchhoff law. Therefore
$e$ is uncovered, and (ii) is proved. (A typical triangle along
with the associated Fano colouring is illustrated in
Figure~\ref{fig:essential3}.)

\begin{figure}[h!]
 \centering
\includegraphics{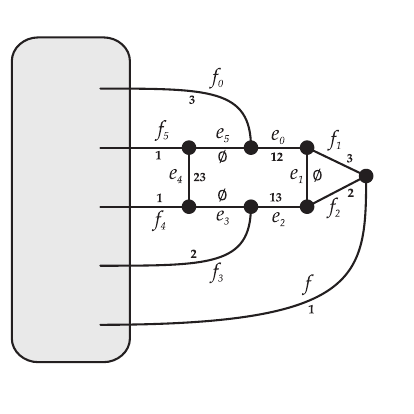}
 \caption{An essential triangle intersected by a hexagonal core}
\label{fig:essential3}
\end{figure}

Next we prove (iii). Suppose to the contrary that
$C=(e_0e_1\ldots e_5)$ intersects two distinct triangles $T_1$
and $T_2$. Adopting the notation of Figure~\ref{fig:core3}, we
can further assume that $T_1$ contains $e_1$ while $T_2$
contains $e_5$, both edges being uncovered. The remaining
uncovered edge $e_3$ may, or may not, belong to a triangle.
Clearly, $T_1=(e_1f_1f_2)$ and $T_2=(e_5f_5f_0)$.
Let us contract each of $T_1$ and $T_2$ to a vertex thereby
producing a cubic graph~$G'=G/(T_1\cup T_2)$. Note that $G'$ is
a snark and $C'=(e_0e_2e_3e_4)$ is a quadrilateral in~$G'$.
Moreover, $G'-V(C')$ is $3$-edge-colourable which means that
$C'$ is non-removable. However, this fact contradicts
Lemma~\ref{lem:odstr4} and establishes (iii).

Assume that $G$ has an essential triangle, and let $T$ be any
of them. If $C$ is an arbitrary hexagonal core, then $C$
intersects all essential triangles, according to (i); in
particular, $C$ intersects $T$. However, (iii) implies that $T$
is the only triangle intersected by $C$. Therefore, $T$ is the
only essential triangle of $G$, which proves (iv). \end{proof}

\begin{lemma}\label{lem:2-cuts}
Let $G$ be a snark with $\df{G}=3$. If $G$ contains a $2$-edge
cut $R$, then the hexagonal core $C$ of any optimal array is
disjoint with $R$. Moreover, $C$ is inherited into the snark
$G'$ with $\df{G'}=3$ arising from the reduction of $G$ with
respect to $R$.
\end{lemma}

\begin{proof}
Let $R$ be a $2$-edge-cut of $G$ whose removal leaves
components $H$ and~$K$. Consider a $3$-array $\mathcal{M}$ for
$G$ whose core is a $6$-cycle $C$ of $G$. We show that $C$ is
wholly contained in one of the components of $G-R$. Suppose
not. Then $C$ intersects $R$, and hence $R\subseteq C$.
Kirchhoff's law for the characteristic flow implies that the
edges of $R$ must have the same colour, which means that both
edges are uncovered. Adopting the notation for $C$ in
accordance with Figure~\ref{fig:core3} we may assume that
$R=\{e_3,e_5\}$. It follows that the edges $e_0$, $e_1$, and
$e_2$ belong to one of the components, say~$H$, and the edge
$e_4$ belongs to the other component. The completion $\bar H$
of $H$ contains a $4$-cycle $(e_0e_1e_2f)$, where $f$ is the
edge resulting from the extension of $H$ to $\bar H$. It is
easy to see that if $f$ is assigned colour~$1$, then the
$3$-edge-colouring of $H-E(C)$ induced by $\mathcal{M}$
uniquely extends to a $3$-edge-colouring of $\bar H$. Similarly
we can check that the completion $\bar K$ of $K$ is
$3$-edge-colourable, too. With both $\bar H$ and $\bar K$ being
$3$-edge-colourable, we conclude that so is~$G$, which is a
contradiction. This proves that $C$ is contained in one of the
components of $G-R$, say~$H$. All the edges of $K$ are now
simply covered, so $K$ is $3$-edge-colourable, and hence $H$ is
not. It is easy to see that the matchings $M_1\cap H$, $M_2\cap
H$, and $ M_3\cap H$ of $H$ extend to perfect matchings $M_1'$,
$M_2'$, and $M_3'$ of $\bar H$, which constitute a $3$-array of
$\bar H$ having $C$ as its core. Therefore $\df{\bar H}=3$, and
$G'=\bar H$ is the sought reduction of $G$.
\end{proof}

\begin{lemma}\label{lem:3-cuts}
Let $G$ be a snark with $\df{G} = 3$. If $G$ contains a
cycle-separating $3$-cut, then $G$ admits a reduction to a
smaller snark $G'$ with $\df{G'}=3$, unless one of the
components resulting from the removal of the cut is an
essential triangle. Every hexagonal core of $G$ that does
not intersect a triangle is inherited into $G'$.
\end{lemma}

\begin{proof}
Let $R=\{r_1,r_2,r_3\}$ be an arbitrary cycle-separating
$3$-edge-cut in $G$, and let $H$ and $K$ be the components of
$G-R$. There are two cases to consider.

\medskip\noindent
Case 1. \emph{$G$ admits an optimal $3$-array
$\mathcal{M}$ such that $\core(\mathcal{M})\cap R=\emptyset$.}
We show that contracting the component of $G-R$ not containing
$\core(\mathcal{M)}$ to a vertex produces a proper reduction
$G'$ of $G$ with $\df{G'}=3$. Let $C$ be the $6$-cycle
constituting the core of $\mathcal{M}$. assumption guarantees
that $C$ is fully contained in a component of $G-R$, say $H$.
Consider the completions $\bar{H}$ and $\bar{K}$ to cubic
graphs. Since $K$ is $3$-edge-colourable, $H$ is not. Therefore
$\bar K$ is $3$-edge-colourable, and $\bar H$ is not. By
Theorem~\ref{thm:main}, $\df{\bar{H}}\ge 3$. As $C\subseteq H$,
the edges of $R$ are simply covered, and by the Kirchhoff law
applied to the characteristic flow for $\mathcal{M}$ they
belong to three distinct members of $\mathcal{M}$. It is now
clear that $\mathcal{M}$ induces a $3$-array $\mathcal{M}'$ of
$\bar H$ with $C$ as its core. Therefore $\df{\bar{H}}= 3$, and
$G'=\bar H$ is the required reduction of $G$ containing $C$.
Note that $G'$ is isomorphic to the graph $G/K$ obtained from
$G$ by contracting $K$ into a single vertex. This establishes
Case~1.

\medskip\noindent
Case 2. \emph{The core of every optimal $3$-array for $G$
intersects the $3$-edge-cut $R$.} We start with the following
observation.

\medskip\noindent
Claim 1. \emph{$G-R$ has a unique component $Q$ such $G/Q$ is
a snark. Moreover, $C\cap Q$ consists of a single uncovered edge. }

\medskip\noindent
\emph{Proof of Claim 1.} Clearly, $|C\cap R| = 2$, and we may
assume that $C\cap R=\{r_1,r_2\}$. Kirchhoff's law yields that
$\chi(r_1)+\chi(r_2)+\chi(r_3)=0$. By
Proposition~\ref{prop:notriply}, $\chi$ is a nowhere-zero
$\mathbb{Z}_2^3$-flow, therefore the values $\chi(r_1)$,
$\chi(r_2)$, and $\chi(r_3)$ constitute a line $\ell$ in the
Fano plane. Since $C$ intersects $R$, the line $\{(0,1,1),
(1,0,1), (1,1,0)\}=\{1,2,3\}$ is excluded. There remain two
possibilities for $\ell$, which imply that either
\begin{itemize}
\item both $r_1$ and $r_2$ are doubly covered, or

\item one of them is doubly covered and the other is
    uncovered.
\end{itemize}
Without loss of generality we may assume that in the former
case $\phi(r_1)=12$, $\phi(r_2)=13$, and
$\phi(r_3)=1$, and in the latter case $\phi(r_1)=12$,
$\phi(r_2)=\emptyset$, and $\phi(r_3)=3$, see
Figure~\ref{fig:konf}. We prove that the latter possibility
does not occur.

Suppose to the contrary that $\phi(r_1)=12$,
$\phi(r_2)=\emptyset$, and $\phi(r_3)=3$. Since the edges
in $R$ are independent, we conclude that $r_1=e_0$ and
$r_2=e_3$, see Figure~\ref{fig:core3}. Let $H$ be the component
of $G-R$ that contains $e_1$ and $e_2$. If we set
$\phi'(e_0)=2$, then the $3$-edge-colouring of $\phi$ of
$H-E(C)$ associated with $\mathcal{M}$ extends to a
$3$-edge-colouring $\phi'$ of~$\bar H$. By symmetry, the
$3$-edge-colouring $\phi$ of $K-E(C)$ extends to a
$3$-edge-colouring of~$\bar K$. It follows that $G$ is
$3$-edge-colourable, which is a contradiction.

Therefore
$\phi(r_1)=12$, $\phi(r_2)=13$, and $\phi(r_3)=1$. Now
$r_1=e_0$ and $r_2=e_2$, so one of the components of $G-R$, say
$H$,  contains the path $e_3e_4e_5$ and the other component $K$
contains the uncovered edge $e_1$. Clearly, $K$ is
$3$-edge-colourable. In other words, $K$ is the required
component $Q$ of $G-R$ such that $Q\cap C$ consists of a single
uncovered edge. This establishes Claim~1.

\medskip
To finish the proof it remains to establish the following.

\medskip\noindent
Claim 2. \emph{If $Q$ is not an essential triangle, then $G$
has a proper reduction to a snark $G'$ with $\df{G'}=3$.}

\medskip\noindent
\emph{Proof of Claim 2.} We keep the assumptions adopted
in the course of the proof of Claim~1. In particular, $C\cap
R=\{e_0,e_2\}$, $Q=K$,  and the unique edge of $Q\cap C$ is $e_1$.

First assume that $Q=K$ is a triangle. Clearly, the graph $G/K$
obtained by contracting $K$ into a single vertex is a snark, so
$\df{G/K}\ge 3$. As $K$ is not essential, we conclude that
$\df{G/K}=3$, implying that $G/K$ is the required reduction of
$G$. Note that, in this case, every hexagonal core of
$G/K$ intersects the vertex of $G/K$ resulting from the
contraction of $K$, which means that it is not inherited from
$G$.

Next assume that $K$ is not a triangle. Consider the
$3$-edge-cut $R'=\{f_1,f_2,r_3\}$ in~$G$. Let $H'$ and $K'$ be
the components of $G-R'$, with $K'$ being the one that does not
contain~$e_1$. Note that $C\subseteq H'$. Clearly,
$\mathcal{M}$ induces a proper $3$-edge-colouring of $K'$, so
$H'$ is not $3$-edge-colourable, and  therefore the graph
$G'=G/K'$ obtained from $G$ by contracting $K'$ into a single
vertex is a snark. Clearly, $G'$ is a proper reduction of $G$.
Observe that the edges $e_1$, $f_1$, and $f_2$ form a triangle
separated by $R'$ from the rest of $G'$. To show that $G'$ is
the sought reduction it remains to check that $\df{G'}=3$. To
this end it is sufficient to realise that $M_i'=M_i\cap G'$ is
a perfect matching of $G'$ for each $i\in\{1,2,3\}$. Thus
$\mathcal{M}'=\{M_1',M_2',M_3'\}$ is a $3$-array of $G'$ with
$\core(\mathcal{M}')=\core(\mathcal{M})=C$, and we are done.
This concludes the proof of Claim~2 as well as that of
Lemma~\ref{lem:3-cuts}.
\end{proof}

\begin{example}\label{ex:non-essential}
Lemma~\ref{lem:essential}(i) informs us that if a snark with
defect $3$ contains an essential triangle, then the triangle
must be intersected by every hexagonal core. Somewhat
surprisingly, the converse is not true, which implies that the
discussion following Claim~2 in the proof of
Lemma~\ref{lem:3-cuts} cannot be avoided. 
The graph $G$ in Figure~\ref{fig:trim_a} has defect $3$ and
possesses three hexagonal cores, all of which intersect its
only triangle $T$. The triangle $T$ is \emph{not} essential in
$G$, because the graph $G/T$, shown in Figure~\ref{fig:trim_b},
has defect $3$ as well. The latter graph has two hexagonal
cores, both containing the vertex resulting from contraction of
$T$, which is denoted by $v_1$. Note that inflating $G/T$ at
any of the vertices $v_i$ with $i\in\{1,2,3,4\}$ produces a
snark with a triangle having the same property as $T$ has in
$G$.
\begin{figure}[h!]
 \centering
\subfigure[]{\includegraphics[width=0.4\textwidth]{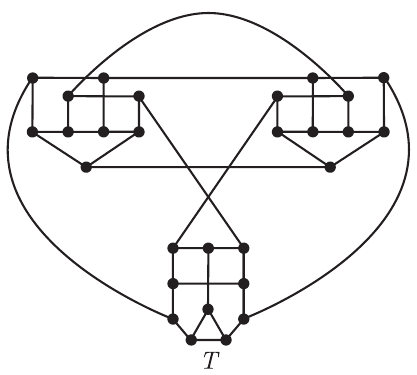}\label{fig:trim_a}}
\hfil
\subfigure[]{\includegraphics[width=0.4\textwidth]{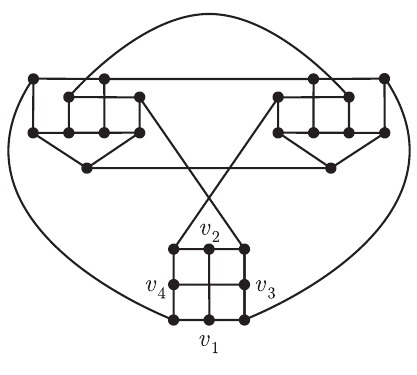}\label{fig:trim_b}}
 \caption{No hexagonal core of the snark $G$ is inherited into $G/T$}
\label{fig:trim}
\end{figure}
\end{example}

\begin{lemma}\label{lem:core_quadrangle}
Let $G$ be a snark with $\df{G}=3$. If a hexagonal core of $G$
intersects a quadrilateral, then the intersection consists of a
single uncovered edge. Moreover, two edges leaving the
quadrilateral are doubly covered and the other two are simply
covered.
\end{lemma}

\begin{proof}
Consider an arbitrary optimal $3$-array $\mathcal{M}$ of $G$,
and let $C$ be its hexa\-gonal core. We keep using the previous
notation for the characteristic flow and the Fano colouring
around $C$ (cf. Figure~\ref{fig:konf} and
Figure~\ref{fig:core3}). Further, let $D$ be a 4-cycle in $G$,
and let $R=\delta(D)$.

\medskip\noindent
Claim 1. \emph{$|C\cap R|=2$.}

\medskip\noindent
\emph{Proof of Claim 1.} Clearly, $|C\cap R|$ is even, so $|C\cap
R|=2$ or $|C\cap R|=4$. The latter possibility cannot occur.
Indeed, if we had $|C\cap R|=4$, then $C\cap D$ would consist
of two independent edges of $D$. However, the other two edges
of $D$ would constitute chords of~$C$, contrary to
Theorem~\ref{thm:main}.

\medskip
By Claim~1, the edges of $R$ naturally come in two pairs, those
that belong to $C$ and those that do not. Let us assume that
$R=\{r_1,r_2,r_3,r_4\}$ where  $C\cap R=\{r_1,r_2\}$. The other
pair $\{r_3,r_4\}$ thus consists of simply covered edges.

\medskip\noindent
Claim 2. \emph{$\phi(r_1)\ne \phi(r_2)$.}

\medskip\noindent
\emph{Proof of Claim 2.} Suppose to the contrary that
$\phi(r_1)=\phi(r_2)$. Clearly, this is only possible
when both $r_1$ and $r_2$ are uncovered. Since $C\cap D$ is a
path and $C$ has no chords, we conclude that $C\cap D$ is not a
path of length $3$. Since we are assuming
$\phi(r_1)=\phi(r_2)$, $C\cap D$ cannot be a path of
length $2$. We conclude that $C\cap D$ has one edge.  Without
loss of generality we may assume that $r_1=e_5$ and $r_2=e_1$,
so the unique edge of $C\cap D$ is $e_0$. It follows that
$D=(e_0f_1d_0f_0)$, where $d_0$ is the edge joining the
end-vertices $u_0$ and $u_1$ of $f_0$ and~$f_1$, respectively,
not lying on $C$. Since $\phi(r_1)=\phi(r_2)$,
Kirchhoff's law implies that $\phi(r_3)=\phi(r_4)$.
Recalling that $\phi(f_0)=\phi(f_1)=3$, we conclude
$\phi(r_3)\in\{1,2\}$. Without loss of generality we may
assume that $\phi(r_3)=\phi(r_4)=1$. We are going to
recolour the edges of $C\cup D$. If we set $\phi'(e_3)=3$,
we can uniquely extend the $3$-edge-colouring of $\phi$ of
$G-E(C\cup D)$ induced by $\mathcal{M}$ to a $3$-edge-colouring
$\phi'$ of the entire $G$ with $\phi'(e_0)=1$,
$\phi'(e_1)=3$, $\phi'(e_2)=1$, $\phi'(e_4)=2$,
$\phi'(e_5)=3$, $\phi'(f_0)=2$, $\phi'(f_1)=2$, and
$\phi'(d_0)=3$. This contradiction proves that
$\phi(r_1)\ne \phi(r_2)$.

\medskip
An important consequence of Claim~2 is that the set
$\ell=\{\chi(r_1),\chi(r_2),\chi(r_3)+\chi(r_4)\}$ forms a line
of the Fano plane. Replacing $\chi$ with $\phi$, there are
two possibilities for $\ell$ up to permutation of the index set
$\{1,2,3\}$, just as in the proof of Lemma~\ref{lem:3-cuts}
(see Claim~1 therein): either $\ell=\{12,13,1\}$ or
$\ell=\{\emptyset,12,3\}$. Next we show that the latter
possibility does not occur.

\medskip\noindent
Claim 3.
\emph{$\{\phi(r_1),\phi(r_2),\phi(r_3)+\phi(r_4)\}
=\{12,13,1\}$.}

\medskip\noindent
\emph{Proof of Claim 3.} Suppose to the contrary that
$\{\phi(r_1),\phi(r_2),\phi(r_3)+\phi(r_4)\}
=\{\emptyset,12,3\}$. In view of symmetry, we can assume that
$\phi(r_1)=12$ and $\phi(r_2)=\emptyset$. It follows that
$C\cap D$ coincides with the path $e_1e_2$ or $e_5e_4$. Without
loos of generality we may assume the former, so
$f_2\in\{r_3,r_4\}$ and $D=(e_1e_2f_3f_1)$. We may further
assume that $f_2=r_3$, whence $\phi(r_3)=2$. Since $r_4$ shares
a common vertex with $f_1$ and $f_3$, and one has $\phi(f_1)=3$
and $\phi(f_3)=2$, we conclude that $\phi(r_4)=1$. This, in
particular, confirms the assumption that
$\phi(r_3)+\phi(r_4)=3$. If we set $\phi'(e_5)=2$, then the
$3$-edge-colouring of $\phi$ of $G-E(C)$ induced by
$\mathcal{M}$ uniquely extends to a $3$-edge-colouring $\phi'$
of the entire $G$ with $\phi'(e_0)=1$, $\phi'(e_1)=3$,
$\phi'(e_2)=1$, $\phi'(e_3)=2$, $\phi'(e_4)=3$, $\phi'(f_1)=2$,
and $\phi'(f_3)=3$. This contradiction establishes Claim~3.

\medskip
We have just proved that
$\{\phi(r_1),\phi(r_2),\phi(r_3)+\phi(r_4)\}
=\{12,13,1\}$. In view of symmetry, we can assume that
$\phi(r_1)=12$, which implies that $r_1=e_0$, $r_2=e_2$, and
that the unique edge of $C\cap D$ is the uncovered edge $e_1$.
This proves the lemma.
\end{proof}

\begin{proposition}\label{lem:4-cycles}
Let $G$ be a snark with $\df{G}=3$. If $G$ contains a
$4$-cycle, then $G$ can be reduced to a smaller snark $G'$ with
$\df{G'}=3$. Moreover, every hexagonal core of $G$ is inherited
into $G'$.
\end{proposition}

\begin{proof}
Let $D$ be a 4-cycle in $G$, and let $R=\delta(D) =
\{r_1,r_2,r_3,r_4\}$. We may assume that the edges of $R$ are
independent, because otherwise $G$ would have a
cycle-separating $2$-cut or $3$-cut, and we could apply
Lemmas~\ref{lem:2-cuts} and~\ref{lem:3-cuts} to conclude that
$G$ admits a proper reduction to a snark $G'$ with $\df{G'}=3$.
Let $\mathcal{M}=\{M_1,M_2,M_3\}$ be an arbitrary optimal
$3$-array of $G$, and let $C$ be its hexagonal core. There
are essentially two possibilities: either $C\cap D=\emptyset$,
or by Lemma~\ref{lem:core_quadrangle}, $C\cap D$ consists of
the unique edge.

If $C\cap D=\emptyset$, then $C\cap R=\emptyset$. By
Lemma~\ref{lem:odstr4}, the graph $G-V(D)$ is not
$3$-edge-colourable. The four dangling edges of $G-V(D)$
are simply covered, and by the Kirchhoff law they occur in two
equally coloured pairs. We join each pair into a single edge
thereby producing a snark $G'$ of order $|G'|=|G|-4$. The Fano
flow on $G$ clearly induces one on $G'$, which in turn
determines the same core $C$. Hence, $G'$ is the required
proper reduction of $G$.

\begin{figure}[h!]
 \centering
\subfigure[]{\includegraphics[width=0.4\textwidth]{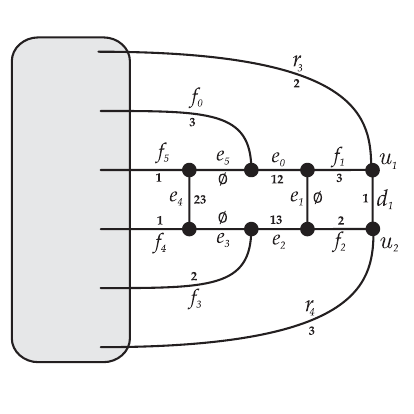}\label{fig:quadrangle_a}}
\hfil
\subfigure[]{\includegraphics[width=0.4\textwidth]{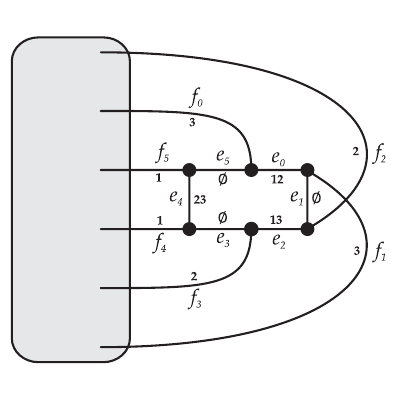}\label{fig:quadrangle_b}}
 \caption{A quadrilateral intersected by a hexagonal core and its reduction}
\label{fig:quadrangle}
\end{figure}

Henceforth we may assume that $C\cap R\ne\emptyset$. From
Lemma~\ref{lem:core_quadrangle} we obtain that the intersection
of the core and the quadrangle consists of a unique
uncovered edge of $C$, say~$e_1$.
Since $r_1=e_0$ and $r_2=e_2$, we conclude that
$D=(e_1f_1d_1f_2)$, where $d_1$ is the edge of $D$ joining the
end-vertices $u_1$ and $u_2$ of $f_1$ and~$f_2$, respectively,
not lying on $C$. Without loss of generality we may assume that
$r_3$  and $r_4$ are incident with $u_1$ and $u_2$,
respectively. Recall that $\phi(f_1)=3$ and
$\phi(f_2)=2$, which implies that $\phi(d_1)=1$,
$\phi(r_3)=2$, and $\phi(r_4)=3$, see
Figure~\ref{fig:quadrangle_a}. Now, we take the graph
$G-\{u_1,u_2\}$, and keep the four dangling edges $f_1$, $f_2$,
$r_3$, and $r_4$. We form $G'$ from $G-\{u_1,u_2\}$ by
performing the junctions $f_1*r_4$ and $f_2*r_3$. Since
$G-V(D)$ is not $3$-edge-colourable, we deduce that $\df{G'}\ge
3$. If we define $\phi'$ by setting
$\phi'(f_1*r_4)=\phi(f_1)=\phi(r_4)=3$,
$\phi'(f_2*r_3)=\phi(f_2)=\phi(r_3)=2$, and
$\phi'(x)=\phi(x)$ for all other edges $x$ of $G'$, we
obtain a Fano colouring which determines a $3$-array of $G'$
whose core coincides with $C$, see Figure~\ref{fig:quadrangle_b}.
\end{proof}

\section{Main results}\label{sec:main}

\noindent{}We are now ready to establish the main results of this paper.

\begin{theorem}\label{thm:reduction}
Every snark $G$ with $\df{G}=3$ admits a reduction to a snark
$G'$ with $\df{G'}=3$ such that either $G'$ is nontrivial or
$G'$ arises from a nontrivial snark $K$ with $\df{K}\ge 4$ by
inflating a vertex to a triangle; the triangle is essential in
both $G$ and~$G'$.
\end{theorem}

\begin{proof}
Consider an arbitrary snark $G$ with $\df{G}=3$. If $G$ is
nontrivial, then $G'=G$ is the required reduction. Assume that
$G$ is not nontrivial, but it cannot be reduced to a
nontrivial snark with $\df{G}=3$. We show that $G$ has a
reduction to a snark $G'$ which arises from a nontrivial snark
$K$ with $\df{K}\ge 4$ by inflating a vertex to a triangle.

Let $G'$ be a reduction of $G$ with $\df{G'}=3$ such that
$G'$ is not nontrivial, but it has no reduction to a smaller
snark with defect $3$. Lemmas~\ref{lem:2-cuts}
and~\ref{lem:4-cycles} imply that $G'$ has no $2$-edge-cuts and
no quadrilaterals. Since $G'$ is not nontrivial, it has a
cycle-separating $3$-edge-cut. By Lemma~\ref{lem:3-cuts}, one
of the resulting components must be an essential triangle,
which we denote by $T$. Since the latter is true for any
cycle-separating $3$-edge-cut in~$G'$, and because $G'$ has a
unique essential triangle according to
Lemma~\ref{lem:essential}, the graph $K=G'/T$ is cyclically
$4$-edge-connected and has $\df{K}\ge 4$. It remains to show
that $K$ has no quadrilaterals.

Suppose to the contrary that $K$ contains a quadrilateral $Q$.
The vertex $K$ obtained by the contraction of $T$ lies on $Q$,
therefore $G'$ has a $5$-cycle $D=(d_0d_2d_1d_3d_4)$ sharing an
edge with $T$, say $d_0$. Let $g$ and $h$ denote the other two
edges of $T$, with $g$ adjacent to $d_1$, and $h$ adjacent to
$d_4$. There is a $4$-edge-cut $S$ in $G'$ that separates $D\cup
T$ from the rest of~$G'$. Let $H$ be the other component of
$G'-S$. Observe that $S$ survives the contraction of~$T$, so
$S$ separates $Q$ from $H$ in $G'/T$ as well. Since $G'/T$  is
cyclically $4$-edge-connected, $S$ is an independent edge cut.
Let $s_1\in S$ be the edge of $T$ not adjacent to $d_0$, and
for $i\in\{2,3,4\}$ let $s_i$ be the edge of $S$ adjacent to
both $d_{i-1}$ and $d_i$, see Figure~\ref{fig:pentagon}.

\begin{figure}[h!]
 \centering
\includegraphics{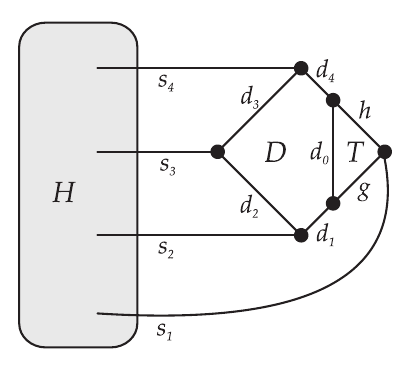}
\caption{The structure of $G'$ when the contraction of an
essential triangle creates a $4$-cycle} \label{fig:pentagon}
\end{figure}

Because $\df{G'}=3$, there is a $3$-array
$\mathcal{M}=\{M_1,M_2,M_3\}$ in $G'$ whose hexagonal core $C$
intersects~$T$. By Lemma~\ref{lem:essential}(ii), $C\cap
T$ consists of a single uncovered edge, so $C$ cannot be the
hexagon $(d_1d_2d_3d_4hg)$ fully contained in $D\cup T$.
Therefore $C$ intersects $S$. It is easy to see that $|C\cap
S|=2$. Without loss of generality we may assume that the common
edge of $C$ and $T$ is the edge $e_1$ of the standard hexagonal
core shown in Figure~\ref{fig:core3}. There are two
possibilities for the position of $e_1$: either $e_1$ is
adjacent to $s_1$, or not. If $e_1$ is not adjacent to $s_1$,
then $e_1=d_0$. If $e_1$ is adjacent to $s_1$, then
$e_1\in\{g,h\}$, and in view of symmetry we may assume that
$e_1=g$. Accordingly, we have two cases to consider.

\medskip\noindent
Case 1. $e_1=d_0$. In this case both $d_1$ and $d_4$ belong to
$C$; in view of symmetry we may clearly assume that $d_1=e_2$
and $d_4=e_0$. Moreover, $s_1$ does not belong to $C$, so
$C\cap S$ consists of two edges from $\{s_2,s_3,s_4\}$. In view
of symmetry we may assume that either $C\cap S=\{s_2,s_3\}$ or
$C\cap S=\{s_2,s_4\}$. The former possibility does not occur
because otherwise the edge $d_2$ would be a chord of $C$,
contrary to Theorem~\ref{thm:main}. Hence $C\cap
S=\{s_2,s_4\}$. Now $s_2=e_3$ and $s_4=e_5$, so the only edge
of $C$ in $H$ is $e_4$. Recall that $\phi(e_4)=23$. If we
change this colour to $2$ (or apply
Proposition~\ref{prop:core}~(iv)), we obtain a proper
$3$-edge-colouring of $H$. However, $H$ is now
$3$-edge-colourable, and the other component of $G'/T-S$ is a
quadrilateral. Therefore $G'/T$ is $3$-edge-colourable, and so
is~$G'$. This contradiction concludes Case~1.

\medskip\noindent
Case 2. $e_1=g$. In this case $d_1$ and $s_1$ belong $C$; in
view of symmetry we may clearly assume that $d_1=e_0$ and
$s_1=e_2$. It follows that $C\cap S=\{s_1,s_j\}$ for some
$j\in\{2,3,4\}$. We show that $j=2$. If $C\cap S=\{s_1,s_3\}$,
then $f_1=d_0$, $f_2=h$, and $f_5=d_3$. Since $\phi(f_1)=3$
and $\phi(f_2)=2$, we have $\phi(d_4)=1$. At the same
time $\phi(f_5)=1$, so $d_3$ and $d_4$ have the same colour,
which is impossible, because they are adjacent. Therefore $j\ne
3$. Next, if $C\cap S=\{s_1,s_4\}$, then $d_4=e_4$ and
$s_4=e_3$. In this situation, however, both $e_2$ and $e_3$
belong to $S$ although they are adjacent in $C$. As $S$ is
independent, this is impossible. Therefore $j\ne 4$, and we
conclude that $C\cap S=\{s_1,s_2\}$. Now $e_2=s_1$ and
$e_5=s_2$, so $e_3$ and $e_4$ belong to $H$. Recall that
$\phi(e_3)=\emptyset$ and $\phi(e_4)=23$. If we change
the colour of $e_3$ to $3$ and the colour of $e_4$ to $2$, we
get a $3$-edge-colouring of $H$. Again, this implies that $G'$
is $3$-edge-colourable, which is a contradiction. This
concludes Case~2 and proves that $K$ has no quadrilaterals. The
proof is complete.
\end{proof}

Our second theorem specifies conditions that must be satisfied
by a snark $K$ and a vertex $v$ in order for the inflation of
$v$ to decrease the defect of $K$ to $3$. The formulation
features an important concept of a cluster of $5$-cycles in a
snark that derives from relatively little known results of
K\'aszonyi \cite{Ka1-ortho}-\cite{Ka3-struct} and Bradley
\cite{Brad1}-\cite{Brad3} concerning the structure of
$3$-edge-colourings of graphs. A \emph{cluster of $5$-cycles}
in a cubic graph $G$, or simply a \emph{$5$-cluster} of $G$, is
an inclusion-wise maximal connected subgraph of $G$ formed by a
union of $5$-cycles. K\'aszonyi \cite{Ka2-nonplan, Ka3-struct}
and later Bradley \cite{Brad1} proved that for each edge $e$ of
a snark $G$ there exists a nonnegative integer $\psi_G(e)$ such
that the number of $3$-edge-colourings of $G\sim e$ equals
$18\cdot\psi_G(e)$. We refer to the function $\psi_G\colon
E(G)\to\mathbb{N}$ as the \emph{K\'aszonyi function} for~$G$.
In passing we mention that the K\'aszonyi function for the
Petersen graph identically equals~$1$, see
\cite[Theorem~3.5]{Brad3}.

One of the most remarkable properties of the K\'aszonyi
function is that it is constant on each $5$-cluster (see
\cite{Brad1}, \cite{Ka3-struct}, or the survey \cite{Brad3}). A
$5$-cluster $H$ of a snark $G$ will be called \emph{heavy} if
$\psi_G(e)>0$ for each edge $e$ of $H$; otherwise, $H$ will be
called \emph{light}. Equivalently, a $5$-cluster $H$ is heavy
if and only if $G\sim e$ is $3$-edge-colourable for each edge
$e$ of $H$. In this context it is useful to recall that if
$e=uv$, then $G\sim e$ is $3$-edge-colourable if and only if
$G-\{u,v\}$ is, see \cite[Proposition~4.2]{NS-decred}.

\begin{theorem}\label{thm:essential}
Let $K$ be a nontrivial snark with $\df{K}\ge 4$, let $v$ be a
vertex of $K$, and let $G$ be the snark created from $K$ by
inflating $v$ to a triangle. Then $\df{G}=3$ if and only if $v$
lies in a heavy cluster of $5$-cycles of $K$.
\end{theorem}

\begin{proof}
$(\Rightarrow)$ Assume that $\df{G}=3$. Let $T$ denote the
triangle of $G$ obtained by the inflation of a vertex $v$ of
$K$. We show that $v$ belongs to a heavy $5$-cluster of $K$.
Since $\df{K}\geq 4$ and $\df{G}=3$, the triangle $T$ is
essential. Let $C$ be a hexagonal core of~$G$.
Lemma~\ref{lem:essential}(ii) implies that $C\cap T$ consists of a
single uncovered edge, which we may assume to be the edge $e_1$
indicated in Figure~\ref{fig:core3}. It follows that the Fano
colouring $\phi$ around $C$ is as illustrated in
Figure~\ref{fig:essential3}. Let us contract $T$ back to the
vertex $v$ and keep the colours of the edges of $K$. Clearly,
$C/T=(e_0e_2e_3e_4e_5)$ is a $5$-cycle containing $v$. To prove
that $v$ belongs to a heavy $5$-cluster it is sufficient to
show that $K\sim e_0$ admits a $3$-edge-colouring. Recall that
$e_0=v_0v_1$ and consider the graph $K-\{v_0,v_1\}$. If we set
$\phi'(e_3)=3$ and $\phi'(e_4)=2$, we obtain a proper
$3$-edge-colouring of $K-\{v_0,v_1\}$. We infer that $K\sim
e_0$ is $3$-edge-colourable, too, and therefore $v$ belongs to
a heavy $5$-cluster of $K$.

$(\Leftarrow)$ For the converse, assume that $v$ belongs to a
heavy $5$-cluster $H$ of $K$. Consider an (induced) $5$-cycle
$D=(d_0d_1d_2d_3d_4)$ of $H$ such that $v$ lies in $D$ and is
incident with the edges $d_4$ and~$d_0$. Since $H$ is heavy,
removing any edge of $D$ together with its end-vertices
produces a $3$-edge-colourable graph. Hence $K-E(D)$ is
$3$-edge-colourable, as well. Let $g_i$ be the edge of
$\delta(D)$ adjacent to $d_{i-1}$ and $d_i$, with indices taken
modulo~$5$. The edge of $\delta(D)$ incident with $v$ is
therefore~$g_0$. From the Parity Lemma we deduce that every
$3$-edge-colouring $\sigma$ of $K-E(D)$ colours three of the
edges in $\delta(D)$ with the same colour, say~$1$, and the
remaining two with colours $2$ and $3$, respectively. Moreover,
if $\sigma(g_i)=2$ and $\sigma(g_j)=3$, then $g_i$ and $g_j$
are not adjacent to the same edge of $D$, otherwise $\sigma$
would extend to a $3$-edge-colouring of $K$. Even more, as can
be deduced from Lemmas~6.1, 6.2, and~6.3~(iii) of
\cite{NS-decred}, for any two edges $g_i$ and $g_j$ of
$\delta(D)$ that are not incident with the same edge of $D$
there exists a $3$-edge-colouring $\tau$ of $K-E(D)$ such that
$\tau(g_i)=2$ and $\tau(g_j)=3$. Set $\tau(g_1)=2$ and
$\tau(g_4)=3$, so that all the remaining edges of $\delta(D)$
receive colour $1$, see Figure~\ref{fig:5to6} (left).
\begin{figure}[h!]
 \centering
\includegraphics[width=0.6\textwidth]{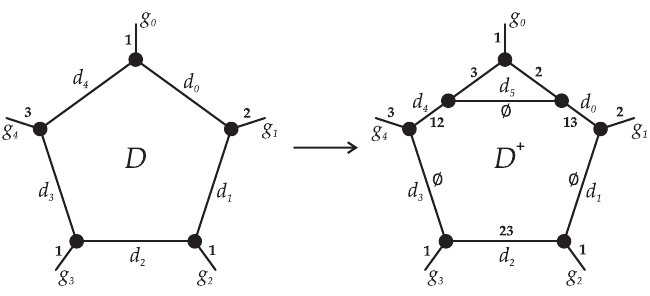}
\caption{Creating an essential triangle by inflating a vertex on a
non-removable $5$-cycle} \label{fig:5to6}
\end{figure}
Now, let
us inflate $v$ into a triangle $T$, thereby producing the graph
$G$. For each edge of $K$ incident with $v$ there is a unique
corresponding edge of $G$ leaving the triangle~$T$; we let the
latter edge keep the name of the former. Having made this
agreement, let $d_5$ denote the edge of $T$ adjacent to both
$d_0$ and $d_4$. Clearly, $D^+=(d_0d_1d_2d_3d_4d_5)$ is an
induced $6$-cycle of $G$. Now we extend the colouring $\tau$ of
$K-E(D)$ to a colouring of $G$. We start by setting
$\tau(d_4)=12$,  $\tau(d_5)=\emptyset$, and $\tau(d_0)=13$.
This choice further enables setting $\tau(d_2)=23$ and
$\tau(d_1)=\tau(d_3)=\emptyset$, as well as assigning colours
$2$ and $3$ to the remaining edges of $T$ appropriately, see
Figure~\ref{fig:5to6} (right). It is easy to check that $\tau$
induces a $3$-array $\mathcal{N}$ of $G$ with
$\core(\mathcal{N})=D^+$. Therefore $\df{G}=3$, as required.
\end{proof}

Given a cubic graph $G$ and a vertex $v$ of $G$ we let $G^v$
denote the graph formed from $G$ by inflating $v$ to a
triangle. We now show that a single vertex inflation can
decrease defect from an arbitrarily large value to the minimal
possible value of $3$.

\begin{theorem}\label{thm:df_n3}
For every integer $n\geq 3$ there exists a nontrivial snark $G$
with $\df{G}\ge n$ which contains a vertex whose inflation to
a triangle produces a snark with defect $3$.
\end{theorem}

\begin{proof}
In \cite[Theorems~5.1-5.3]{KMNS-girth} it was proved that for
every integer $n\ge3$ there exists a cyclically
$5$-edge-connected snark $H=H_{2n}$ with girth $2n$ which
contains a pair of adjacent vertices $u$ and $v$ such that
$H-\{u,v\}$ is $3$-edge-colourable. By \cite[Corollary~2.5]{S2}
(see also \cite[Proposition~4.4]{KMNS-girth}), the defect of
$H$ is at least $n$. Let $I$ denote the isochromatic
$(2,2)$-pole arising from $H_{2n}$ by removing the vertices $u$
and $v$ and forming connectors from the semiedges formerly
incident with the same vertex. To construct $G=G_n$, we take
three copies $I_0$, $I_1$, and $I_2$ of $I$ and a $6$-pole $Z$
arising from the Petersen graph $Pg$ by severing three
independent edges $p_0$, $p_1$, and $p_2$ on a $6$-cycle. We
turn $Z$ into a $(2,2,2)$-pole with connectors $S_1$, $S_2$,
and $S_3$ by forming each $S_i$ from the two half-edges of
$p_i$, denoted by $p_{i1}$ and $p_{i2}$. Next, for each
$j\in\{0,1,2\}$,  we join the input connector of $I_j$ to
$S_j$; we keep the notation $p_{j1}$ and $p_{j2}$ for the
resulting edges which connect $I_j$ to $Z$. Finally, we match
the semiedges of the three output connectors in such a way that
the two semiedges of each output connector lead to two other
output connectors, see Figure~\ref{fig:gn_cluster}(left). The
result is the required graph $G=G_n$.

We proceed to proving that $G_n$ has the required properties.
During our analysis we refer to
Figure~\ref{fig:gn_cluster}(right) for the notation of vertices
and edges of $G_n$. In particular, $z$ denotes the central
vertex of $Z$, the edges incident with $z$ are $e_0$, $e_1$,
and $e_2$, and $(f_0f_1\ldots f_8)$ is the $9$-cycle obtained
by removing $z$ from $Z$. The edges of $\delta(Z)$ leave $Z$ in
the order $p_{01}, p_{22}, p_{11}, p_{02}, p_{21}, p_{12}$
determined by a cyclic orientation of $Z-z$.

\medskip\noindent{}Claim 1. $G_n$ is a nontrivial snark.

\medskip
\noindent\emph{Proof of Claim 1.} First we prove that
$G_n$ is a snark. Suppose that $G_n$ admits a
$3$-edge-colouring $\phi$. Since each $I_j$ is an
isochromatic $(2,2)$-pole, the edges $p_{j1}$ and $p_{j2}$
receive the same colour from every $3$-edge-colouring of $G_n$.
Recall, that $p_{j1}$ and $p_{j2}$ arise by severing the edge
$p_j$ of $Pg$. It follows that the restriction of $\phi$ to
$Z$ extends to a $3$-edge-colouring of $Pg$, which is a
contradiction. Thus $G_n$ is a snark.

It is clear from the construction that $G_n$ has girth at least
$5$. We need to check that $G_n$ is cyclically
$4$-edge-connected. To this end, it suffices to realise that
the building blocks of $G_n$ -- the $6$-pole $Z$ and the
$4$-poles $I_0$, $I_1$, and $I_2$ -- arise from nontrivial
snarks and that the way in which the building blocks have been
combined in $G_n$ guarantees that no cycle-separating edge cut of
size smaller than $4$ can be created. Therefore every
cycle-separating edge cut in $G_n$ has size at least $4$.  In
fact, all minimum size cycle-separating edge cuts in $G_n$
separate one of the $4$-poles $I_j$ from the rest of $G_n$;
hence, the cyclic connectivity of $G_n$ equals $4$.

\begin{figure}[h!]
\centering
\subfigure{
\includegraphics[width=0.45\textwidth]{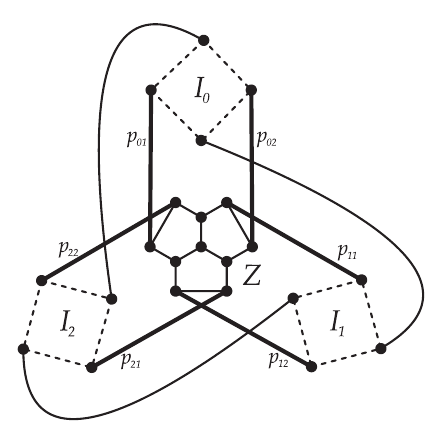}}
\subfigure{
\includegraphics[width=0.45\textwidth]{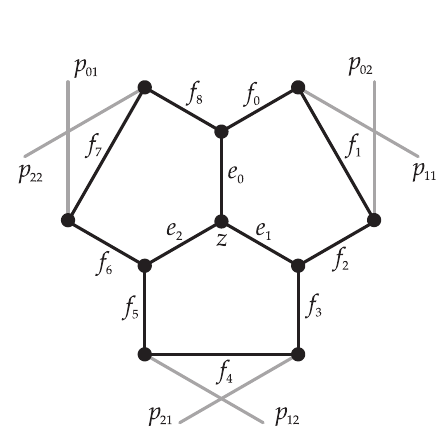}}
\caption{The graph $G_n$ and its heavy cluster $Z$}
\label{fig:gn_cluster}
\end{figure}

\medskip

\noindent{}Claim 2. $G_n$ has no $3$-array whose core is fully
contained in $Z$.

\medskip\noindent
\emph{Proof of Claim 2.} Suppose to the contrary that $G_n$ has a
$3$-array $\mathcal{M}$ with core $C\subseteq Z$. It follows
that all the edges of $G_n-Z$ as well as those belonging to the
edge cut $\delta(Z)$ connecting $Z$ to $G_n-Z$ are simply
covered. There are two cases to consider.

\medskip\noindent
Case 1.\emph{ $C$ contains a triply covered edge.} If $x$ is a
triply covered edge of $C$, then the four edges adjacent to $x$
are uncovered, so $x$ cannot be adjacent to an edge of
$\delta(Z)$. Therefore $x$ is one of the edges $e_0$, $e_1$, or
$e_2$ incident with the central vertex $z$ of $Z$, say $e_0$.
Assume that the $9$-cycle $(f_0f_1\ldots f_8)$ has its edges
enumerated cyclically in such a way that $f_0$ is adjacent to
$e_0$ and $f_2$ is adjacent to $e_1$, see
Figure~\ref{fig:gn_cluster}(right). All four edges adjacent to
$e_0$ are uncovered, in particular, so are $f_0$ and $e_1$.
Since the edges of $\delta(Z)$ are simply covered, it follows
that $f_1$ is doubly covered and $f_2$ is uncovered. Now, $e_1$
and $f_2$ are adjacent uncovered edges, so $f_3$ must be triply
covered. However, $f_3$ is adjacent to an edge of $\delta(Z)$,
which is simply covered, and we have arrived at a
contradiction.

\medskip\noindent
Case 2.\emph{ $C$ contains no triply covered edge.}  In this
case the core $C$ is regular, and hence, by
Proposition~\ref{prop:core}~(i), it is a collection of disjoint
even circuits. Observe that every circuit contained in $Z$ is
either a pentagon, an $8$-gon, or a $9$-gon. Therefore $C$ must
be a single $8$-gon. Without loss of generality we may assume
that $C=(e_0f_0f_1f_2f_3f_4f_5e_2)$. Now, one of $e_0$ and
$e_2$ must be uncovered, say $e_0$. It follows that $f_0$ and
$f_4$ are doubly covered and $f_8$ is simply covered. Consider
the Fano colouring $\phi$ of $G_n$ associated with the $3$-array
$\mathcal{M}$. Without loss of generality we may assume that
$\phi(f_0)=12$. This implies that $\phi(f_8)=\phi(p_{11})=3$,
and since $I_1$ is isochromatic, we infer that
$\phi(p_{12})=\phi(p_{11})=3$. Note that $f_4$ is doubly
covered and is adjacent to both $\phi(p_{12})$ and
$\phi(p_{21})$, so $\phi(f_4)=12$, and hence $\phi(p_{21})=3$.
Using the isochromatic property of $I_2$ we conclude that
$\phi(p_{22})=3$, which is impossible because $p_{22}$ is
adjacent to $f_8$ and $\phi(f_8)=3=\phi(p_{22})$. This
contradiction completes the proof of Claim~2.

\medskip
\noindent{}Claim 3. $\df{G_n}\ge n$.

\medskip\noindent
\emph{Proof of Claim 3.} Let $D$ be a circuit of the core of
any $3$-array of $G_n$. By Claim~2, $D$ must intersect at least
one of $I_0$, $I_1$, and $I_2$. If $D$ is contained in some
$I_j$, then its length is clearly at least $2n$. Otherwise, $D$
contains at least $2n-2$ vertices of $I_j$ and at least two
vertices of $Z$, and again its length is at least $2n$. Recall
that the edges of $D$ are of three kinds -- uncovered, doubly
covered and triply covered. Moreover, by
Proposition~\ref{prop:core}(iii), the union of doubly and
triply covered edges in $D$ forms a matching of $D$. Therefore,
there are at least $n$ uncovered edges in $D$. In other words,
the defect of $G_n$ is at least $n$.

\medskip
\noindent{}Claim 4. \emph{The $5$-cluster $Z$ is heavy.}

\medskip\noindent
\emph{Proof of Claim 4.} It is sufficient to show that $G_n\sim
x$ is $3$-edge-colourable for some edge $x$ of $Z$, say
$x=e_0$. Let $J$ denote the $6$-pole obtained from $G_n$ by
removing the vertices of the $6$-pole $Z$, so that $G_n=J*Z$.
Recall that every isochromatic $4$-pole admits a
$3$-edge-colouring where all four dangling edges receive the
same colour, see for example \cite[Section~3]{ChS}). It follows
that $J$ admits a $3$-edge-colouring where all six dangling
edges receive the same colour. It is easy to check that such an
assignment of colours to the dangling edges of $Z$ extends to a
$3$-edge-colouring of $Z\sim e_0$; we leave the details to the
reader. By combining these two $3$-edge-colourings we obtain
one for $G_n\sim e_0$. This proves that $Z$ is a heavy
$5$-cluster.

\medskip

Now we can finish the proof. Claim~4 states that $\df{G_n}\ge
n$. On the other hand, Theorem~\ref{thm:essential} implies that
the inflation of every vertex $v$ of $Z$  produces a graph
$G_n^v$ with $\df{G_n^v}=3$. Both required properties of $G_n$
are verified, and the proof is complete.
\end{proof}

\begin{figure}[h!]
\centering
\includegraphics[width=0.42\textwidth]{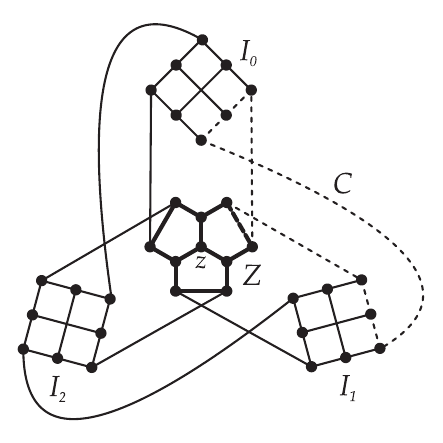}
\caption{The smallest nontrivial snark with $\text{df}\ge 4$ that contains a
heavy $5$-cluster}
\label{fig:snark34_43}
\end{figure}

\begin{example}~\label{ex:df43}
The smallest example of a nontrivial snark with defect
greater than $3$ containing a vertex whose inflation produces a
snark with defect $3$ has $34$ vertices; it is depicted in
Figure~\ref{fig:snark34_43}. It structure is similar to the
graphs constructed in Theorem~\ref{thm:df_n3}. The isochromatic
$(2,2)$-poles $I_0$, $I_1$, and $I_2$ arise from the Petersen
graph by removing two adjacent vertices. The defect of this
snark is $4$, and the corresponding core $C$ is an $8$-cycle
indicated in Figure~\ref{fig:snark34_43} by dashed edges.
\end{example}

\section{Berge covers of snarks with defect~3}
\label{sec:Berge}

\noindent{}To justify the importance of the results of the
previous two sections, we briefly indicate how they apply to
verifying Berge's conjecture for snarks with defect $3$. We
show that every bridgeless cubic graph with defect $3$ can have
its edges covered with four or five perfect matchings and we
determine those that require five. In other words, we
completely determine the perfect matching index of defect $3$
graphs. A detailed proof will appear in \cite{KMNS-Bergegen}.
Recall that the \emph{perfect matching index} (also known as
the \textit{excessive index}) of a bridgeless cubic graph $G$,
denoted by $\pi(G)$, is the smallest number of perfect
matchings that cover all the edges of~$G$. With this
definition, Berge's conjecture states that $\pi(G)\le 5$ for
every bridgeless cubic graph $G$. Note that $\pi(G)\ge 3$, and
the equality holds if and only if $G$ is $3$-edge-colourable.

\medskip

The following theorem is proved in \cite{KMNS-Berge}.

\begin{theorem}\label{thm:4-ec}
Let $G$ be a cyclically $4$-edge-connected cubic graph with
defect~$3$. Then $\pi(G)=4$, unless $G$ is the Petersen graph.
\end{theorem}

In order to be able to discuss the general situation in the
class of defect $3$ graphs we will make use of the following
two well-known operations. Let $G$ and $H$ be cubic graphs with
distinguished edges $e$ and $f$, respectively. We define a
$2$-sum $G\oplus_2 H$ to be a cubic graph obtained by deleting
$e$ and $f$ and connecting the $2$-valent vertices of $G$ to
those of~$H$. If instead of distinguished edges we have
distinguished vertices $u$ and $v$ of $G$ and~$H$,
respectively, we can similarly define a $3$-sum $G\oplus_3 H$.
We simply remove $u$ and $v$ and join the $2$-valent vertices
of $G-u$ to those of $H-v$ with three independent edges. Note
that $G\oplus_3 H$ can be regarded as being obtained from $G$
by inflating the vertex $u$ to $H-v$.

A cubic graph $G$ containing a cycle-separating $2$-edge-cut or
$3$-edge-cut can be expressed as $G_1\oplus_2 G_2$ or
$G_1\oplus_3 G_2$ uniquely, only depending on the chosen
edge-cut. It is easy to see that if two $2$-edge-cuts or
$3$-edge-cuts intersect, the result of decomposition does not
depend on the order in which the cuts are taken. As a
consequence, we have the following.

\begin{theorem}\label{thm:decomposition}
Every $2$-connected cubic graph $G$ admits a decomposition into
a collection $\{G_1,\ldots,G_m\}$ of cyclically
$4$-edge-connected cubic graphs such that $G$ can be
reconstructed from them by repeated application of $2$-sums
and $3$-sums. Moreover, this collection is unique up to
ordering and isomorphism.
\end{theorem}

Theorems~\ref{thm:4-ec} and ~\ref{thm:decomposition}, combined
with results of the previous sections and with
\cite[Theorem~4.1]{KMNS-Berge} (see also
\cite[Theorem~2.1]{KMNS-eurocomb}) can now be used to prove
that cubic graphs of defect~$3$ fulfil Berge's conjecture.

\begin{theorem}
Every $2$-connected cubic graph $G$ with colouring defect $3$
has $\pi(G)=4$ or $\pi(G)=5$. Moreover, if $G$ has an essential
triangle, then $\pi(G)=4$.
\end{theorem}

In order to characterise the cubic graphs of defect $3$ that
have perfect matching index equal to $5$ we need to introduce a
new concept. A bridgeless cubic graph $Q$ is
\emph{quasi-bipartite} if it contains an independent set of
vertices $U$ such that the graph obtained by the contraction of
each component of $Q-U$ to a vertex is a cubic bipartite graph
where $U$ is one of the partite sets. Roughly speaking, a
quasi-bipartite cubic graph arises from a bipartite cubic graph
by inflating certain vertices in one of the partite sets to
larger subgraphs, while preserving the edges between the
partite sets.

The next theorem describes conditions under which a $3$-sum of
two graphs has perfect matching index at least $5$. A $3$-sum
with one of the summands being quasi-bipartite will be called
\emph{correct} if the resulting graph is again quasi-bipartite.

\begin{theorem}\label{thm:correct}
Let $G$ and $H$ be $2$-connected cubic graphs with
distinguished vertices $u$ and~$v$, respectively, where
$\pi(G)\ge 5$ and $H$ is $3$-edge-colourable. Assume that
$\pi(G-u)=4$. Then $\pi(G\oplus_3 H)\ge 5$ if and only if $H$
is quasi-bipartite and the $3$-sum is correct.
\end{theorem}

By Lemma~\ref{lem:2-cuts}, no hexagonal core can intersect a
$2$-edge-cut. Applying Theorem~\ref{thm:decomposition} we can
now conclude that every $2$-connected cubic graph with defect
$3$ arises from a $3$-connected cubic graph $H$ with defect $3$
by performing repeated $2$-sums of $H$ with $3$-edge-colourable
graphs in such a way that a core of $H$ is not affected by the
$2$-sums. It follows that we can restrict ourselves to
$3$-connected graphs.

The final result, for $3$-connected graphs, reads as follows.
Its proof involves the use of Theorem~\ref{thm:4-ec}, the
decomposition into cyclically $4$-edge-connected graphs
presented in Theorem~\ref{thm:decomposition}, and a repeated
application of Theorem~\ref{thm:correct}.

\begin{theorem}\label{thm:bergegen}
Every $2$-connected cubic graph $G$ of defect $3$ has perfect
matching index at most $5$. If $G$ is $3$-connected, then
$\pi(G)=5$ if and only if $G$ arises from the Petersen graph by
inflating any number of vertices of a fixed vertex-star by
quasi-bipartite cubic graphs in a correct~way.
\end{theorem}

\section{Concluding remarks}\label{sec:comp}
\noindent{}Here we analyse the defect and several related
invariants of small snarks. Our analysis is computer-aided. We
have computed the defect of all cyclically $4$-edge-connected
snarks of girth at least $5$ and of order at most $36$ from the
database \emph{House of Graphs: Snarks}~\cite{HoG}. We
summarise the output in Table~\ref{tab:defects}. As expected,
the major part of the analysed snarks (in fact, around $99.999089\%$)
have defect $3$. The defect of all nontrivial snarks with at
most $36$ vertices takes values in the set $\{3,4,5,6\}$.

We shall briefly discuss the smallest nontrivial snarks of
defect $4$, $5$, and $6$ in more detail.
\begin{figure}[h!]
\subfigure[$\df{G_{28}}=5$]{
\includegraphics[width=0.28\textwidth]{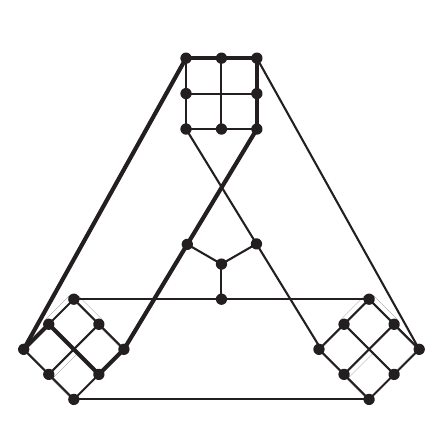}\label{fig:G28}}\quad
\subfigure[$\df{G_{32}}=4$]{ \includegraphics[width=0.28\textwidth]{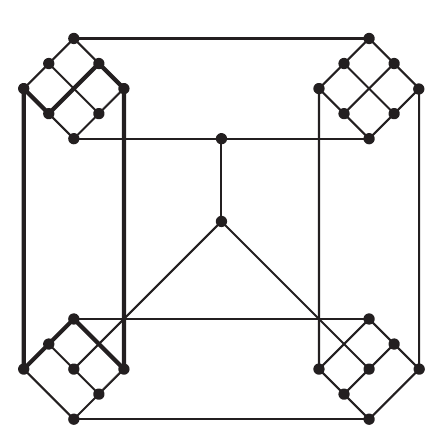}\label{fig:G32}}\quad
\subfigure[$\df{G_{34}}=6$]{ \includegraphics[width=0.28\textwidth]{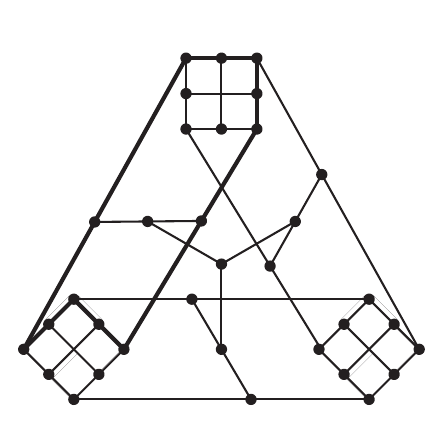}\label{fig:G34} }
\caption{Smallest nontrivial snarks with defect greater than $3$}
\label{Gbigdf}
\end{figure}
The smallest nontrivial snark of defect greater than $3$ has
order $28$. The graph is denoted by $G_{28}$ and is displayed
in Fig.~\ref{fig:G28}; it has defect $5$. It is not difficult
to understand the reason. If the defect of $G_{28}$ was $3$ or
$4$, then the core of an optimal $3$-array would be a circuit
of length $6$ or $8$, respectively. The graph $G_{28}$ contains
nine $6$-cycles and three $8$-cycles. However, it can be easily
seen that each of the $6$-cycles and $8$-cycles is removable.
The graph $G_{28}$ appears also in another context in \cite{MR}
as the smallest nontrivial snark different from the Petersen
graph with circular flow number equal to~$5$. The smallest
nontrivial  snark with defect $4$, denoted by $G_{32}$, has
order $32$, see Fig~\ref{fig:G32}. The smallest order where
cyclically $4$-edge-connected snarks with defect $6$ occur is
$34$; there are exactly seven such snarks.  The most
symmetrical of them, denoted by $G_{34}$, is depicted in
Fig.~\ref{fig:G34}. By coincidence, $G_{34}$ is the unique
smallest nontrivial snark different from the Petersen graph
whose edges cannot be covered with four perfect matchings,
see~\cite{GGHM}. In each of the graphs displayed in
Figure~\ref{Gbigdf} bold edges highlight the core of an optimal
$3$-array.

\begin{table}[h!]
\begin{tabular}{|r|rr|rrrr|}
\hline
\hline
\rule{0pt}{11pt}Order & Nontrivial & Critical & $\operatorname{df}=3$ & $\operatorname{df}=4$ & $\operatorname{df}=5$ &$\operatorname{df}=6$ \\

\hline
\rule{0pt}{11pt}10 & 1 & 1 & 1 & - & - & - \\
18 & 2 & 2 & 2 & - & - & - \\
20 & 6 & 1 & 6 & - & - & - \\
22 & 20 & 2 &  20 & - & - & - \\
24 & 38 & - & 38 & - & - & - \\
26 & 280 & 111 & 280  & - & - & -\\
28 & 2900 & 33 & 2899  & - & 1 & - \\
30 & 28399 & 115 & 28397  & - & 2 & -\\
32 & 293059 & 29 & 293049  & 1 & 9 & - \\
34 & 3833587 & 40330 &  3833538 & 24 & 18 & 7 \\
36 & 60167732 & 14548 &  60167208 & 195 & 304 & 25 \\
\hline
$\sum$\rule{0pt}{11pt} & 64326024 & 55172 & 64325438 & 220 & 334 & 32 \\
\hline
\hline
\end{tabular}
\bigskip
\caption{Defects of small nontrivial snarks}
\label{tab:defects}
\end{table}

It transpires that among the nontrivial snarks of order up to
$36$ there are exactly three graphs of defect greater than $3$
with a heavy cluster of $5$-cycles. The smallest one is
depicted in Figure~\ref{fig:snark34_43}; it has $34$ vertices
and defect $4$. The remaining two have order $36$ and defect
$4$ and $5$, respectively. Recall that, by
Theorem~\ref{thm:essential}, the inflation of any vertex in the
heavy cluster decreases the defect to $3$.

\medskip

We conclude this section with three remarks.

\begin{table}[h!]
\begin{tabular}{|r|rr|rrrr|}
\hline
\hline
\multirow{3}{*}{\scriptsize{Order}} & \multirow{3}{*}{\scriptsize{Nontrivial}} & \multirow{3}{*}{\scriptsize{Critical}} & \scriptsize{Double-core} & \scriptsize{Double-core} & \scriptsize{Double-core} & \scriptsize{Double-core} \\
 &  &  &  & \scriptsize{Removable} & \scriptsize{Single-core} & \scriptsize{Single-core} \\
  &  &  &  &  &  & \scriptsize{Removable} \\
\hline
10 & 1 & 1 & 1 & - & - & - \\
18 & 2 & 2 & 2 & - & - & - \\
20 & 6 & 1 & 1 & - & 5 & - \\
22 & 20 & 2 & 3 & - & 17 & - \\
24 & 38 & - & 1 & 6 & 22 & 9 \\
26 & 280 & 111 & 112 & 63 & 21 & 84 \\
28 & 2900 & 33 & 126 & 706 & 1374 & 693 \\
30 & 28399 & 115 & 907 & 9126 & 10798 & 7566 \\
32 & 293059 & 29 & 3693 & 133046 & 53799 & 102511 \\
34 & 3833587 & 40330 & 55144 & 2095876 & 192684 & 1489834 \\
\hline
$\sum$\rule[2pt]{0pt}{9pt} & 4158292 & 40624 &  59990 & 2238823 & 258720 & 1600697\\
\hline
\hline
\end{tabular}
\bigskip

\caption{Hexagons in nontrivial snarks with defect $3$}
\label{tab:6cycles}
\end{table}

\begin{remark}
We have investigated properties of $6$-cycles of all nontrivial
snarks of order not exceeding~$34$. A $6$-cycle $C$ in a snark
can be either removable or non-removable. If $C$ is
non-removable, then one of the following three possibilities
occurs: $C$ is a double-core hexagon, $C$ is a single-core
hexagon, or $C$ does not constitute a hexagonal core. If the
latter occurs, $C$ is a \emph{non-core hexagon}. Outputs of
computations are summarised in Table~\ref{tab:6cycles}. The
four columns in the right-hand side of the table represent a
partition of the set of all nontrivial snarks of defect $3$
with at most $34$ vertices according to the properties of their
$6$-cycles. The column with heading ``Double-core'' contains
the numbers of nontrivial snarks in which every $6$-cycle is a
double-core hexagon for some $3$-array. The column headed by
``Double-core'' and ``Removable''  enumerates nontrivial snarks
in which every $6$-cycle is either a double-core hexagon or it
is removable, and both types of $6$-cycles occur. The families
of snarks enumerated in the remaining two columns are defined
in a similar manner.

Every snark in the collection of tested snarks with defect $3$
has been found to have at least one double-core hexagon, which
is a remarkable phenomenon. This property, however, 
cannot be expected from trivial snarks. The graph $G$ from 
Example~\ref{ex:non-essential}, shown in Figure~\ref{fig:trim_a}, 
contains a triangle that is intersected by all core hexagons. 
According to Lemma~\ref{lem:essential}(ii), each hexagonal core of $G$ 
is single-core. It is therefore natural to ask the following 
question.

\begin{problem}
Does there exist a nontrivial snark with defect $3$ in
which every core hexagon is single-core?
\end{problem}

This problem is particularly interesting from the point of view
of Fulkerson's conjecture. If such a snark did exist, then
either its Fulkerson cover would not consist of two
complementary optimal $3$-arrays, or else the snark would
provide a counterexample to Fulkerson's conjecture.
\end{remark}

\begin{remark}
There exist many snarks where every $6$-cycle is a double-core
hexagon, see Table~\ref{tab:6cycles}, and critical snarks of
order at most $36$ are among them. This observation motivates
the following conjecture.

\begin{conjecture}\label{con:crit}
Every critical snark has defect $3$.
\end{conjecture}

Recall that the existence of a double-core hexagon in a snark
implies that the union of the corresponding two $3$-arrays
constitutes a Fulkerson cover. Therefore, we propose the
following.

\begin{conjecture}\label{con:critful}
In a critical snark, every hexagon is double-core. In
particular, every optimal $3$-array of perfect matchings
extends to a Fulkerson cover.
\end{conjecture}

In \cite{NS-decred} it is proved that irreducible snarks
coincide with bicritical ones. It means that the removal of any
pair of distinct vertices yields a $3$-edge-colourable graph.
Irreducible snarks thus constitute a subfamily of critical
snarks. Hence, we have the following weaker conjecture.

\begin{conjecture}\label{con:irre}
Every irreducible snark $G$ has defect $3$.
\end{conjecture}

If Conjecture~\ref{con:irre} is confirmed, then the following
long-standing conjecture proposed in \cite{NS-decred} is
verified as well.

\begin{conjecture}\label{con:ns}
There exist no irreducible snarks of girth greater than $6$.
\end{conjecture}

We note that Conjecture~\ref{con:ns} can be viewed as an
``improved'' version of once famous girth conjecture for
snarks. Jaeger \cite[Conjecture~1]{JSw} conjectured that there
exist no (nontrivial) snarks of girth greater than 6, which was
later disproved by Kochol in \cite{Ko}.

\medskip

Conjectures~\ref{con:crit} to \ref{con:ns} are related as
follows:
\bigskip
\begin{center}
Conjecture~\ref{con:critful} $\Rightarrow$
Conjecture~\ref{con:crit}    $\Rightarrow$
Conjecture~\ref{con:irre}    $\Rightarrow$
Conjecture~\ref{con:ns}
\end{center}
\bigskip
In~\cite[Conjecture~4.1]{S2} Steffen conjectured that every
hypohamiltonian snark has defect~$3$. Since every
hypohamiltonian snark is easily seen to be
irreducible~\cite{S1}, the validity of
Conjecture~\ref{con:irre} would imply that of Steffen's
conjecture.
\end{remark}

\begin{remark}
In the collection of tested snarks, every non-removable hexagon
in a nontrivial snark is either single-core or double-core. In
other words, non-removable non-core $6$-cycles do not occur
among the tested snarks. The following question suggests
itself.

\begin{problem}
Does there exist a snark $G$ of defect $3$ which contains a
$6$-cycle $C$ such that $G-V(C)$ is $3$-edge-colourable, but
$C$ does not constitute a hexagonal core? (In other words, does
there exist a snark containing a non-removable non-core
hexagon?)
\end{problem}

This problem is closely related to a problem discussed in
\cite{KMN}. It asks whether a certain theoretically derived
colouring set, denoted by $\mathcal{E}_{13}$, admits a
realisation by a suitable $6$-pole. Answering this problem
would represent a significant step towards the so-called
$6$-decomposition theorem for snarks.
\end{remark}

\bigskip

\end{document}